\documentclass[12pt]{amsart}
\usepackage{amsmath}
\usepackage{amsfonts}
\usepackage{amsthm}
\usepackage{amsbsy}
\usepackage{amssymb}
\usepackage{amsthm}
\usepackage{enumerate}
\usepackage[margin=0.55in]{geometry}
\usepackage{hyperref}
\usepackage{mathrsfs}
\usepackage{mathtools}
\usepackage{microtype}
\usepackage[nobysame]{amsrefs}

\numberwithin{equation}{section}
\newtheorem{thm}{Theorem}[section]
\newtheorem{lem}[thm]{Lemma}
\newtheorem{prop}[thm]{Proposition}
\newtheorem{cor}[thm]{Corollary}

\theoremstyle{definition}
\newtheorem{defn}[thm]{Definition}

\newtheorem{rem}[thm]{Remark}
\newtheorem{exam}[thm]{Example}

\newcommand{\bC}{{\mathbb{C}}}
\newcommand{\bE}{{\mathbb{E}}}

\newcommand{\bN}{{\mathbb{N}}}
\newcommand{\bR}{{\mathbb{R}}}

\newcommand{\A}{{\mathcal{A}}}
\newcommand{\B}{{\mathcal{B}}}
\newcommand{\C}{{\mathcal{C}}}

\newcommand{\I}{{\mathcal{I}}}
\newcommand{\K}{{\mathcal{K}}}
\renewcommand{\L}{{\mathcal{L}}}
\newcommand{\M}{{\mathcal{M}}}
\newcommand{\N}{{\mathcal{N}}}

\newcommand{\X}{{\mathcal{X}}}

\newcommand{\ep}{\varepsilon}

\newcommand{\qand}{\quad\text{and}\quad}
\newcommand{\qqand}{\qquad\text{and}\qquad}

\newcommand{\alg}{\mathrm{alg}}

\newcommand{\Cov}{{\rm Cov}}
\newcommand{\bim}{\rhd\!\!\rhd}

\usepackage{tikz}
\usetikzlibrary{shapes,snakes,calc,arrows}
\usetikzlibrary{decorations.pathreplacing,shapes.geometric}
\usetikzlibrary{calc,positioning}
\tikzset{Box/.style={very thick, rounded corners}}
\tikzset{marked/.style={star, star point height = .75mm, star points =5, fill=black,minimum size=2mm, inner sep=0mm} }
\tikzset{verythickline/.style = {line width=7pt}}
\tikzset{thickline/.style = {line width=5pt}}
\tikzset{medthick/.style = {line width=3pt}}
\tikzset{med/.style = {line width=2pt}}
\tikzset{count/.style = {fill=white,circle,draw,thin, inner sep=2pt}}
\tikzset{rcount/.style = {fill=white,rectangle,draw,thin,inner sep=2pt, rounded corners}}
\tikzset{cpr/.style = {draw,fill=white,rectangle,thin, rounded corners}}

\definecolor{ggreen}{HTML}{00BB33}

\begin{document}

\nocite{*}

\title[Bi-monotonic independence for pairs of algebras]{Bi-monotonic independence for pairs of algebras} 

\author{Yinzheng Gu, Takahiro Hasebe, and Paul Skoufranis}

\address{Department of Mathematics and Statistics, Queen's University, Jeffery Hall, Kingston, Ontario, K7L 3N6, Canada} 
\email{gu.y@queensu.ca}

\address{Department of Mathematics, Hokkaido University, North 10, West 8, Kita-ku, Sapporo 060-0810, Japan}
\email{thasebe@math.sci.hokudai.ac.jp}

\address{Department of Mathematics and Statistics, York University, 4700 Keele Street, Toronto, Ontario, M3J 1P3, Canada}
\email{pskoufra@yorku.ca}

\date{\today}

\begin{abstract}
In this article, the notion of bi-monotonic independence is introduced as an extension of monotonic independence to the two-faced framework for a family of pairs of algebras in a non-commutative space. The associated cumulants are defined and a moment-cumulant formula is derived in the bi-monotonic setting. In general the bi-monotonic product of states is not a state and the bi-monotonic convolution of probability measures on the plane is not a probability measure. This provides an additional example of how positivity need not be preserved under conditional bi-free convolutions.
\end{abstract}

\maketitle

\section{Introduction and motivation}

In non-commutative probability theory, many notions of independence for algebras have been introduced in various contexts. According to the classification work \cites{M2002, M2003, S1997}, only five of them (namely tensor, free, Boolean, monotonic, and anti-monotonic) are universal/natural in a certain sense. Out of the five independences, the rule for tensor independence is identical to the classical notion of independence, and the rule for anti-monotonic independence is essentially the same as the one for monotonic independence upon reversing the order structure on random variables. Thus we shall focus on the remaining three notions of independence, which for a family $\{\A_k\}_{k \in K}$ of algebras in a non-commutative probability space $(\A, \varphi)$ are given as follows: 
\begin{enumerate}[$\quad(1)$]
\item the family $\{\A_k\}_{k \in K}$ is \textit{freely independent} with respect to $\varphi$ if
\[
\varphi(a_1\cdots a_n) = 0
\]
whenever $a_j \in \A_{k_j}$, $k_j \in K$, $k_j \neq k_{j+1}$, and $\varphi(a_j) = 0$ for all $j$;
\item the family $\{\A_k\}_{k \in K}$ is \textit{Boolean independent} with respect to $\varphi$ if
\[
\varphi(a_1\cdots a_n) = \varphi(a_1)\cdots\varphi(a_n)
\]
whenever $a_j \in \A_{k_j}$, $k_j \in K$, and $k_j \neq k_{j+1}$ for all $j$;
\item assuming $K$ is equipped with a linear order $<$, the family $\{\A_k\}_{k \in K}$ is \textit{monotonically independent} with respect to $\varphi$ if
\[
\varphi(a_1\cdots a_{p - 1}a_p a_{p + 1}\cdots a_n) = \varphi(a_p)\varphi(a_1\cdots a_{p - 1}a_{p + 1}\cdots a_n)
\]
whenever $a_j \in \A_{k_j}$, $k_j \in K$ for all $j$, and $k_{p - 1} < k_p > k_{p + 1}$, where one of the inequalities is eliminated if $p = 1$ or $p = n$.
\end{enumerate}

Given a non-commutative probability space $(\A, \varphi)$, elements of $\A$ are called non-commutative random variables and are said to be freely/Boolean/monotonically independent with respect to $\varphi$ if the algebras they generate in $\A$ are freely/Boolean/monotonically respectively.  If in addition $\A$ is a unital C$^*$-algebra, $\varphi$ is a state, and $a \in \A$ is self-adjoint, then the distribution of $a$ with respect to $\varphi$ can be identified as a compactly supported probability measure $\mu_a$ on $\bR$. If $a_1 , a_2 \in \A$ are self-adjoint and freely/Boolean/monotonically independent with respect to $\varphi$, then the distribution $\mu_{a_1 + a_2}$ of $a_1 + a_2$ with respect to $\varphi$ depends only on $\mu_{a_1}$ and $\mu_{a_2}$ and is said to be the additive free/Boolean/monotonic convolution of $\mu_{a_1}$ and $\mu_{a_2}$, denoted $\mu_{a_1} \boxplus \mu_{a_2}$, $\mu_{a_1} \uplus \mu_{a_2}$, and $\mu_{a_1} \rhd \mu_{a_2}$ respectively. 

To describe the above convolutions in terms of $\mu_{a_1}$ and $\mu_{a_2}$, define the Cauchy transform of $a_j$ by $G_{a_j}(z) = \varphi((z - a_j)^{-1})$ for $z \in \bC^+$. It is known (see \cite{V1986}) that $G_{a_j}$ is invertible at $\infty$. Let $K_{a_j}$ be the inverse under composition of $G_{a_j}$ so that $K_{a_j}(0) = \infty$, and let $F_{a_j}$ be the reciprocal of $G_{a_j}$. Define $R_{a_j}(z) = K_{a_j}(z) - \frac{1}{z}$ on a neighbourhood of $0$ and $E_{a_j}(z) = z - F_{a_j}(z)$ for $z \in \bC^+$. Then 
\begin{align*}
R_{a_1 + a_2}(z) &= R_{a_1}(z) + R_{a_2}(z) & & \text{if $a_1$ and $a_2$ are freely independent \cite{V1986},} \\
E_{a_1 + a_2}(z) &= E_{a_1}(z) + E_{a_2}(z)& & \text{if $a_1$ and $a_2$ are Boolean independent \cite{SW1997}, and} \\
F_{a_1 + a_2}(z) &= F_{a_1}(F_{a_2}(z)) & & \text{if $a_1$ and $a_2$ are monotonically independent \cite{M2000}.} 
\end{align*}

Recently, Voiculescu extended the notion of free independence to bi-free independence for a family $\{(\A_{k, \ell}, \A_{k, r})\}_{k \in K}$ of pairs of algebras in a non-commutative probability space $(\A, \varphi)$ in \cite{V2014,V2016-1}. In particular, bi-free independence leads to additive convolutions on probability measures on $\bR^2$. More precisely, if $(a, b)$ is a pair of commuting self-adjoint operators in a C$^*$-non-commutative probability space $(\A, \varphi)$, then the joint distribution of $(a, b)$ with respect to $\varphi$ can be identified as a compactly supported probability measure $\mu_{(a, b)}$ on $\bR^2$. If $(a_1, b_1)$ and $(a_2, b_2)$ are two such pairs which are bi-freely  independent with respect to $\varphi$, then $\mu_{(a_1 + a_2, b_1 + b_2)}$ depends only on $\mu_{(a_1, b_1)}$ and $\mu_{(a_2, b_2)}$ and is called the additive bi-free convolution of $\mu_{(a_1, b_1)}$ and $\mu_{(a_2, b_2)}$, denoted $\mu_{(a_1, b_1)} \boxplus\!\boxplus \mu_{(a_2, b_2)}$. Likewise, the notion of Boolean independence has been extended to bi-Boolean independence in \cite{GS2017}.  However, in the case of bi-Boolean independence, $\varphi$ need not be a state because bi-Boolean product of states is not a state in general. This point was omitted in the paper \cite{GS2017} and correspondingly some results are not correct (see the errata to \cite{GS2017}). However, we can still define a bi-Boolean convolution purely combinatorially, namely in terms of moments.

To linearize bi-free and bi-Boolean convolutions, the two-variable analogues of $R$- and $E$-transforms were introduced as follows. Define the two-variable Cauchy transform of $(a, b)$ by $G_{(a, b)}(z, w) = \varphi((z - a)^{-1}(w - b)^{-1})$. Then define 
\begin{align}
R_{(a, b)}(z, w) &= zR_a(z) + wR_b(w) + \widetilde{R}_{(a, b)}(z, w) \qand \notag \\
E_{(a, b)}(z, w) &= \frac{1}{z}E_a(z) + \frac{1}{w}E_b(w) + \widetilde{E}_{(a, b)}(z, w), \label{bi-energy}
\end{align}
where
\[
\widetilde{R}_{(a, b)}(z, w) = 1 - \frac{zw}{G_{(a, b)}(K_a(z), K_b(w))} \qand \widetilde{E}_{(a, b)}(z, w) = \frac{G_{(a, b)}(z, w)}{G_a(z)G_b(w)} - 1.
\]
In particular, $R_{(a, b)}$ and $E_{(a, b)}$ have the properties that
\begin{align*}
R_{(a_1 + a_2, b_1+ b_2)}(z,w) &= R_{(a_1,b_1)}(z,w) + R_{(a_2,b_2)}(z,w) & & \text{if $(a_1,b_1)$ and $(a_2,b_2)$ are bi-freely independent \cite{V2016-1}, and} \\
E_{(a_1 + a_2, b_1+ b_2)}(z,w) &= E_{(a_1,b_1)}(z) + E_{(a_2,b_2)}(z) & & \text{if $(a_1,b_1)$ and $(a_2,b_2)$ are bi-Boolean independent \cite{GS2017}.}
\end{align*}

It is reasonable to investigate whether the notion of monotonic independence has a suitable extension for pairs of algebras with a corresponding additive convolution. There are two natural candidates for such an extension, which are described in Section \ref{sec:bi-mono} below and referred to as bi-monotonic independence of types I and II. These two types are not equivalent and the second type was recently and independently discovered and studied in \cite{G2017}.  We prefer the first type as it is more consistent and in-line with observations and results in bi-free probability.  One example of this is described as follows and other examples are provided in Section \ref{sec:bi-mono}. 

In \cite{BBGS2017} it was noticed that the notions of bi-free and bi-Boolean independences for pairs of commuting self-adjoint operators can be studied in the operator-valued framework for free and Boolean independences. In this setting, a C$^*$-non-commutative probability space $(\A, \varphi)$ is replaced by a C$^*$-$\B$-non-commutative probability space $(\M, \bE, \B)$, where $\M$ is a unital C$^*$-algebra, $\B \subset \M$ is a unital C$^*$-subalgebra, and $\bE: \M \to \B$ is a linear, positive, unit-preserving conditional expectation. The notions of free, Boolean, and monotonic independences with amalgamation over $\B$ with respect to $\bE$ are defined in a similar way to the scalar-valued case. For a self-adjoint operator $X \in \M$, for $b \in \B$ define 
\[
R_X(b) = K_X(b) - b^{-1}, \quad E_X(b) = b - G_X(b)^{-1}, \qand F_X(b) = G_X(b)^{-1},
\]
where $G_X(b) = \bE((b - X)^{-1})$ for $\|b^{-1}\| < \frac{1}{\|X\|}$ or $\Im(b) > 0$, and $K_X$ is the inverse under composition of $G_X$ on a neighbourhood of $0$ in $\B$. Then 
\begin{align*}
R_{X_1 + X_2}(b) &= R_{X_1}(b) + R_{X_2}(b) & & \text{if $X_1$ and $X_2$ are freely independent over $\B$ \cite{V1995},} \\
E_{X_1 + X_2}(b) &= E_{X_1}(b) + E_{X_2}(b) & & \text{if $X_1$ and $X_2$ are Boolean independent over $\B$ \cite{P2009}, and} \\
F_{X_1 + X_2}(b) &= F_{X_1}(F_{X_2}(b)) & & \text{if $X_1$ and $X_2$ are monotonically independent over $\B$ \cite{P2008}.} 
\end{align*}

Consider a pair $(a, b)$ of commuting self-adjoint operators in a C$^*$-non-commutative probability space $(\A, \varphi)$.  Let 
$X = \begin{bmatrix}
a & 0\\
0 & b
\end{bmatrix}$ 
which is a $M_2(\bC)$-valued self-adjoint random variable in the C$^*$-$M_2(\bC)$-non-commutative probability space $(M_2(\A), M_2(\varphi), M_2(\bC))$, where $M_2(\varphi)$ denotes the map $\varphi \otimes \mathrm{Id}_{M_2(\bC)}$ from $\A\otimes M_2(\bC) $ onto $M_2(\bC)$. A crucial observation in \cite{BBGS2017} is that if we consider the $M_2(\bC)$-valued transforms $R_X$ and $E_X$ of $X$ on the $2 \times 2$ upper triangular matrix $\begin{bmatrix}
z & \zeta\\
0 & w
\end{bmatrix} \in M_2(\bC)$, then
\[R_X\left(\begin{bmatrix}
z & \zeta\\
0 & w
\end{bmatrix}\right) = \begin{bmatrix}
R_a(z) & \frac{\zeta}{zw}\widetilde{R}_{(a, b)}(z, w)\\
0 & R_b(w)
\end{bmatrix} \qand E_X\left(\begin{bmatrix}
z & \zeta\\
0 & w
\end{bmatrix}\right) = \begin{bmatrix}
E_a(z) & -\zeta\widetilde{E}_{(a, b)}(z, w)\\
0 & E_b(w)
\end{bmatrix}.\]
In other words, if $(a_1, b_1)$ and $(a_2, b_2)$ are pairs of commuting self-adjoint operators and $X_j$ denotes the $2 \times 2$ diagonal matrix with $a_j, b_j$ as entries, then
$R_{(a_1 + a_2, b_1 + b_2)}(z, w) = R_{(a_1, b_1)}(z, w) + R_{(a_2, b_2)}(z, w)$ if and only if 
\[R_{X_1 + X_2}\left(\begin{bmatrix}
z & \zeta\\
0 & w
\end{bmatrix}\right) = R_{X_1}\left(\begin{bmatrix}
z & \zeta\\
0 & w
\end{bmatrix}\right) + R_{X_2}\left(\begin{bmatrix}
z & \zeta\\
0 & w
\end{bmatrix}\right),
\]
and the same relation holds between the $E$-transforms. Moreover, one can verify that
\[
F_{X_1 + X_2}\left(\begin{bmatrix}
z & \zeta\\
0 & w
\end{bmatrix}\right) = \begin{bmatrix}
F_{a_1 + a_2}(z) & \zeta F_{a_1 + a_2}(z)F_{b_1 + b_2}(w)G_{(a_1 + a_2, b_1 + b_2)}(z, w)\\
0 & F_{b_1 + b_2}(w)
\end{bmatrix}
\]
and
\begin{align*}
F_{X_1}&\left(F_{X_2}\left(\begin{bmatrix}
z & \zeta\\
0 & w
\end{bmatrix}\right)\right) \\
&= \begin{bmatrix}
F_{a_1}(F_{a_2}(z)) & \zeta F_{a_2}(z)F_{b_2}(w)G_{(a_2, b_2)}(z, w)F_{a_1}(F_{a_2}(z))F_{b_1}(F_{b_2}(w))G_{(a_1, b_1)}(F_{a_2}(z), F_{b_2}(w))\\
0 & F_{b_1}(F_{b_2}(w))
\end{bmatrix}.
\end{align*}

It turns out in Subsection \ref{sec:convolution} that if $(a_1, b_1)$ and $(a_2, b_2)$ are bi-monotonically independent with respect to $\varphi$ in the type I sense, then both $a_1, a_2$ and $b_1, b_2$ are monotonically independent with respect to $\varphi$ (so $F_{a_1 + a_2}(z) = F_{a_1}(F_{a_2}(z))$ and $F_{b_1 + b_2}(z) = F_{b_1}(F_{b_2}(z))$), and 
\begin{equation}\label{eqn:biMonCauchy}
G_{(a_1 + a_2, b_1 + b_2)}(z, w) = G_{(a_1, b_1)}(F_{a_2}(z), F_{b_2}(w))G_{(a_2, b_2)}(z, w)F_{a_2}(z)F_{b_2}(w),
\end{equation}
and thus
\[
F_{X_1 + X_2}\left(\begin{bmatrix}
z & \zeta\\
0 & w
\end{bmatrix}\right)  = F_{X_1}\left(F_{X_2}\left(\begin{bmatrix}
z & \zeta\\
0 & w
\end{bmatrix}\right)\right).
\]
This consistency with other bi-probability theories is some of the evidence that the type I bi-monotonic independence is a more favourable extension of monotonic independence for pairs of algebras.
On the other hand it should be noted that type I bi-monotonic product does not preserve states in general, while the type II bi-monotonic product does. 

Besides this introduction, this paper contains three sections which are organized as follows. In Section \ref{sec:bi-mono}, we define the two types of bi-monotonic independence as mentioned above. We introduce a bi-monotonic product for a family of pairs of algebras and use it to define the type I notion. We also prove the convolution formula \eqref{eqn:biMonCauchy}. In Section \ref{sec:cumulants}, we introduce the bi-monotonic cumulants and prove a bi-monotonic moment-cumulant formula via combinatorics. The formula is given by summing over bi-monotonic partitions which naturally extends the monotonic moment-formula to the pairs of algebras setting. Finally, we show that bi-monotonic (and therefore general conditionally bi-free) convolutions do not preserve probability measures on $\bR^2$ in Section \ref{sec:analytic}.

\section{Bi-monotonic independence}\label{sec:bi-mono}

In this section, two types of bi-monotonic independence are introduced. The first notion is defined using the conditionally bi-free product similar to the relation between monotonic and conditionally free products.  The second notion is defined by considering the left and right actions of operators on Muraki's monotonic product space similar to the original definition of bi-free independence.

Throughout the paper we call $(\A, \varphi)$ a non-commutative space if $\A$ is an algebra and $\varphi$ is a linear functional on $\A$, assuming $\varphi(1_\A)=1$ if $\A$ is unital. If $\A$ is a unital $^*$-algebra and $\varphi$ is a state, i.e.\ a unital positive linear functional, then we call $(\A,\varphi)$ a non-commutative probability space. We call $(\A, \varphi,\psi)$ a double non-commutative space if $\A$ is an algebra and $\varphi$ and $\psi$ are linear functionals, which are unital if $\A$ is unital.

\subsection{Bi-monotonic independence of type I}

To begin recall that free, Boolean, and monotonic independences each corresponds to a universal construction in the following sense. Let $\{\A_k\}_{k \in K}$ be a family of algebras and for each $k \in K$ let $\varphi_k: \A_k \to \bC$ be a linear functional ($\varphi_k(1) = 1$ if $\A_k$ is unital). Then there exists a pair $(\A, \varphi)$ such that every $\A_k$ is a subalgebra of $\A$, $\varphi|_{\A_k} = \varphi_k$, and (depending on the construction) $\{\A_k\}_{k \in K}$ is freely, Boolean, or monotonically independent with respect to $\varphi$. We denote $(\A, \varphi)$ by $*_{k \in K}(\A_k, \varphi_k)$, $\diamond_{k \in K}(\A_k, \varphi_k)$, and $\rhd_{k \in K}(\A_k, \varphi_k)$ respectively and $\varphi$ by $*_{k \in K}\varphi_k$, $\diamond_{k \in K}\varphi_k$, and $\rhd_{k \in K}\varphi_k$ respectively.

It turns out that these three products can be unified in terms of another product, called conditionally free (c-free for short) product \cite{BLS1996}. More precisely, suppose $\{\A_k\}_{k \in K}$ is a family of unital algebras such that each $\A_k$ is equipped with a pair $(\varphi_k, \psi_k)$ of unital linear functionals and $\A_k$ decomposes as $\A_k = \bC 1 \oplus \A_k^\circ$ with $\A_k^\circ = \ker(\psi_k)$.   Let $\A =  *_{k \in K}\A_k$ be the algebraic free product of $\{\A_k\}_{k \in K}$ with identification of units.  Then the c-free product of $\{(\varphi_k, \psi_k)\}_{k \in K}$, denoted $(\varphi, \psi) = *_{k \in K}(\varphi_k, \psi_k)$, is the pair of unital linear functionals on $\A$ defined by
\[
\psi(a_1\cdots a_n) = 0 \qand \varphi(a_1\cdots a_n) = \varphi_{k_1}(a_1)\cdots\varphi_{k_n}(a_n)
\]
for all $n \geq 1$, $a_j \in \A_{k_j}$, $k_j \in K$, $k_j \neq k_{j+1}$, and $\psi_{k_j}(a_j) = 0$. By construction, $\varphi|_{\A_k} = \varphi_k$ and $\psi|_{\A_k} = \psi_k$ for all $k \in K$, and $\psi = *_{k \in K}\psi_k$.  Notice that if $\varphi_k = \psi_k$ for all $k \in K$, then $\varphi = \psi$. 

If $\{\A_k\}_{k \in K}$ is a family of algebras each equipped with a linear functional $\varphi_k$, then the Boolean product $\diamond_{k \in K}\varphi_k$ can be realized in terms of c-free product as follows. For every $k \in K$, let $\widetilde{\A_k} := \bC 1 \oplus \A_k$ be the unitization of $\A_k$, let $\widetilde{\varphi_k}$ be the unique unital linear extension of $\varphi_k$ to $\widetilde{\A_k}$, and let $\delta_k$ denote the delta functional on $\widetilde{\A_k}$ defined by $\delta_k(\lambda 1 + a_0) = \lambda$ for $\lambda \in \bC$ and $a_0 \in \A_k$. Let $\widetilde{\A} = *_{k \in K}\widetilde{\A_k}$ and consider the c-free product $(\varphi, \delta) = *_{k \in K}(\widetilde{\varphi_k}, \delta_k)$ on $\widetilde{\A}$.  Then $\widetilde{\A} \simeq \bC 1 \oplus \A$, where $\A = \sqcup_{k \in K}\A_k$ is the free product of $\{\A_k\}_{k \in K}$ without identification of units, and $\varphi|_\A = \diamond_{k \in K}\varphi_k$. Note that both the free product and the Boolean product are associative, which can be either shown directly or follow from the fact that the c-free product is associative.

It was shown by Franz in \cite{F2006} that the monotonic product can also be realized as a c-free product. For simplicity, we restrict to the case $K = \{1 < 2\}$. Let $(\A_1, \varphi_1)$ and $(\A_2, \varphi_2)$ be as in the Boolean case with unitizations and extensions $(\widetilde{\A_1}, \widetilde{\varphi_1})$ and $(\widetilde{\A_2}, \widetilde{\varphi_2})$ respectively, and let $\delta_1$ be the delta functional on $\widetilde{\A_1}$. Moreover, let $\widetilde{\A} = \widetilde{\A_1} * \widetilde{\A_2} \simeq \widetilde{\A_1 \sqcup \A_2}$ and consider the c-free product $(\varphi, \psi) = (\widetilde{\varphi_1}, \delta_1) * (\widetilde{\varphi_2}, \widetilde{\varphi_2})$ on $\widetilde{\A}$. Then it follows from \cite{F2006}*{Proposition 3.1} that $\varphi|_{\A_1 \sqcup \A_2} = \varphi_1 \rhd \varphi_2$. Note \cite{F2001} demonstrated that the monotonic product is associative which does not follow from the associativity of the c-free product due to the asymmetry in $(\widetilde{\varphi_1}, \delta_1)$ and $(\widetilde{\varphi_2}, \widetilde{\varphi_2})$. 

We now turn our attention to the pairs of algebras setting. The notion of bi-free independence was introduced in \cite{V2014} to study the left and right actions on a reduced free product space simultaneously. Moreover, this new notion of independence is also universal in the sense that if $\{(\A_{k, \ell}, \A_{k, r})\}_{k \in K}$ is a family of pairs of unital algebras such that for each $k \in K$ there is a unital linear functional $\varphi_k: \A_{k, \ell} * \A_{k, r} \to \bC$, then \cite{V2014}*{Corollary 2.10} implies that there is a unique unital linear functional $\varphi: *_{k \in K}(\A_{k, \ell} * \A_{k, r}) \to \bC$ such that $\varphi|_{\A_{k, \ell} * \A_{k, r}} = \varphi_k$ and $\{(\A_{k, \ell}, \A_{k, r})\}_{k \in K}$ is bi-freely independent with respect to $\varphi$. The functional $\varphi$ is called the bi-free product of $\{\varphi_k\}_{k \in K}$ and is denoted by $\varphi = *\!*_{k \in K}\varphi_k$. 

The notions of conditionally bi-free (c-bi-free for short) and bi-Boolean independences were introduced in \cites{GS2016, GS2017} as generalizations of c-free and Boolean independences and are universal. For c-bi-free, given a pair of unital linear functionals $\varphi_k, \psi_k: \A_{k, \ell} * \A_{k, r} \to \bC$ for each $k \in K$ there is a unique pair of unital linear functionals $\varphi, \psi: *_{k \in K}(\A_{k, \ell} * \A_{k, r}) \to \bC$ such that $\varphi|_{\A_{k, \ell} * \A_{k, r}} = \varphi_k$, $\psi|_{\A_{k, \ell} * \A_{k, r}} = \psi_k$, and $\{(\A_{k, \ell}, \A_{k, r})\}_{k \in K}$ is c-bi-freely independent with respect to $(\varphi, \psi)$. The pair $(\varphi, \psi)$  is called the c-bi-free product of $\{(\varphi_k, \psi_k)\}_{k \in K}$ and is denoted by $(\varphi, \psi) = *\!*_{k \in K}(\varphi_k, \psi_k)$, where $\psi = *\!*_{k \in K}\psi_k$ is the bi-free product of $\{\psi_k\}_{k \in K}$. The bi-Boolean product can be realized in terms of the c-bi-free product by taking unitizations and the delta functionals in the same way as the relation between the Boolean and c-free products. For more details, see \cite{GS2017}*{Section 3}.

We now define the bi-monotonic product of linear functionals along the same lines as how the monotonic product can be realized in terms of the c-free product. Consider first the case that $K = \{1 < 2\}$.

\begin{defn}\label{def:bimonotone}
Let $(\A_{1, \ell}, \A_{1, r})$ and $(\A_{2, \ell}, \A_{2, r})$ be two pairs of algebras with linear functionals $\varphi_k: \A_{k, \ell} \sqcup \A_{k, r} \to \bC$. The \textit{bi-monotonic product} of $\varphi_1$ and $\varphi_2$, denoted $\varphi = \varphi_1 \bim \varphi_2$, is the linear functional on $(\A_{1, \ell} \sqcup \A_{1, r}) \sqcup (\A_{2, \ell} \sqcup \A_{2, r})$ defined as follows. Let $\widetilde{\A_{k, \ell}}$ and $\widetilde{\A_{k, r}}$ be the unitizations of $\A_{k, \ell}$ and $\A_{k, r}$ respectively, let $\widetilde{\varphi_k}$ be the unique unital linear extension of $\varphi_k$ to $\widetilde{\A_{k, \ell}} * \widetilde{\A_{k, r}} \simeq \widetilde{\A_{k, \ell} \sqcup \A_{k, r}}$, let $\delta_1$ be the delta functional on $\widetilde{\A_{1, \ell} \sqcup \A_{1, r}}$, and let $(\widetilde{\varphi}, \widetilde{\psi}) = (\widetilde{\varphi_1}, \delta_1) *\!* (\widetilde{\varphi_2}, \widetilde{\varphi_2})$ be the c-bi-free product on $\widetilde{\A_{1, \ell} \sqcup \A_{1, r}} * \widetilde{\A_{2, \ell} \sqcup \A_{2, r}} \simeq \widetilde{(\A_{1, \ell} \sqcup \A_{1, r}) \sqcup (\A_{2, \ell} \sqcup \A_{2, r})}$. Then $\varphi$ is defined to be the restriction of $\widetilde{\varphi}$ to $(\A_{1, \ell} \sqcup \A_{1, r}) \sqcup (\A_{2, \ell} \sqcup \A_{2, r})$.
\end{defn}

The associativity of the bi-monotonic product does not automatically follow from that of the c-bi-free product due to the antisymmetry of the functionals and thus must be demonstrated.  To begin we  review some notation used in bi-free and bi-Boolean probabilities (see \cites{CNS2015-1, CNS2015-2, GS2017}).

Given $n \geq 1$ and elements $a_1, \dots, a_n$ in a non-commutative space,  $V = \{v_1 < \cdots < v_s\} \subset \{1, \dots, n\}$, we denote
\[
(a_1, \dots, a_n)|_V := (a_{v_1}, \dots, a_{v_s}) \qand a_V := a_{v_1}\cdots a_{v_s}.
\]
For a map $\chi: \{1, \dots, n\} \to \{\ell, r\}$ with $\chi^{-1}(\{\ell\}) = \{i_1 < \cdots < i_p\}$ and $\chi^{-1}(\{r\}) = \{i_{p + 1} > \cdots > i_n\}$, define a permutation $s_\chi$ on $\{1, \dots, n\}$ by $s_\chi(j) = i_j$ for $1 \leq j \leq n$, and define a total order $\prec_\chi$ on $\{1, \dots, n\}$ by $i_1 \prec_\chi \cdots \prec_\chi i_n$. A subset $V \subset \{1, \dots, n\}$ is said to be a \textit{$\chi$-interval} if it is an interval with respect to $\prec_\chi$. In addition, we define $\min_{\prec_\chi}(V)$ and $\max_{\prec_\chi}(V)$ to be the minimal and maximal elements of $V$ with respect to $\prec_\chi$, respectively. 

Given $n \geq 1$, $\chi: \{1, \dots, n\} \to \{\ell, r\}$, and $\omega: \{1, \dots, n\} \to K$, let $\pi_{\chi, \omega}$ be the unique partition of $\{1, \dots, n\}$ with ordered blocks $V_1, \dots, V_m$ such that each $V_k$ is a $\chi$-interval, $\max_{\prec_\chi}(V_k) \prec_\chi \min_{\prec_\chi}(V_{k + 1})$, $\omega$ is constant on each $V_k$, and $\omega(V_k) \neq \omega(V_{k + 1})$ for all $1 \leq k \leq m - 1$.

For example, if $\chi^{-1}(\{\ell\}) = \{1, 3, 5, 7, 8, 11\}$ and $\chi^{-1}(\{r\}) = \{2, 4, 6, 9, 10, 12\}$, then the above total order $\prec_\chi$ is obtained by writing each number from $1,\dots,12$  either on the left or on the right of a folded line in a falling manner, and then by unfolding the line. It is  given by 
$$
1 \prec_\chi 3 \prec_\chi  5 \prec_\chi  7 \prec_\chi  8 \prec_\chi 11 \prec_\chi 12 \prec_\chi 10\prec_\chi 9\prec_\chi 6\prec_\chi 4\prec_\chi 2. 
$$ 
If furthermore $K = \{a, b\}$, $\omega^{-1}(\{a\}) = \{1, 2, 3, 4, 6, 8, 11, 12\}$ and $\omega^{-1}(\{b\}) = \{5, 7, 9, 10\}$, then the partition $\pi_{\chi, \omega}$ is given by 
\[
\pi_{\chi, \omega} = \{\{1, 3\}, \{2, 4, 6\}, \{5, 7\}, \{8, 11, 12\}, \{9, 10\} \}.
\]

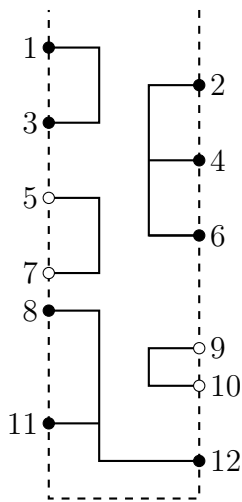
\begin{align*}
\begin{tikzpicture}[baseline]
	\draw[thick, dashed] (1, 6) -- (1, -0.5) -- (-1, -0.5) -- (-1, 6);
	\draw[thick] (-1, 5.5) -- (-0.33, 5.5) -- (-0.33, 4.5) -- (-1, 4.5);
	\draw[thick] (-1, 3.5) -- (-0.33, 3.5) -- (-0.33, 2.5) -- (-1, 2.5);
	\draw[thick] (-1, 2) -- (-0.33, 2) -- (-0.33, 0) -- (1, 0);
	\draw[thick] (-1, 0.5) -- (-0.33, 0.5);
	\draw[thick] (1, 5) -- (0.33, 5) -- (0.33, 3) -- (1, 3);
	\draw[thick] (1, 4) -- (0.33, 4);
	\draw[thick] (1, 3) -- (0.33, 3);
	\draw[thick] (1, 1.5) -- (0.33, 1.5) -- (0.33, 1) -- (1, 1);
	\node[left] at (-1, 5.5) {$1$};
	\draw[fill=black] (-1,5.5) circle (0.075);
	\node[right] at (1, 5) {$2$};
	\draw[fill=black] (1,5) circle (0.075);
	\node[left] at (-1,4.5) {$3$};
	\draw[fill=black] (-1,4.5) circle (0.075);
	\node[right] at (1, 4) {$4$};
	\draw[fill=black] (1,4) circle (0.075);
	\node[left] at (-1, 3.5) {$5$};
	\draw[fill=white] (-1,3.5) circle (0.075);
	\node[right] at (1, 3) {$6$};
	\draw[fill=black] (1,3) circle (0.075);
	\node[left] at (-1, 2.5) {$7$};
	\draw[fill=white] (-1,2.5) circle (0.075);
	\node[left] at (-1, 2) {$8$};
	\draw[fill=black] (-1,2) circle (0.075);
	\node[right] at (1, 1.5) {$9$};
	\draw[fill=white] (1,1.5) circle (0.075);
	\node[right] at (1, 1) {$10$};
	\draw[fill=white] (1,1) circle (0.075);
	\node[left] at (-1, .5) {$11$};
	\draw[fill=black] (-1,.5) circle (0.075);
	\node[right] at (1,0) {$12$};
	\draw[fill=black] (1,0) circle (0.075);
\end{tikzpicture}
\end{align*}

The above definition was used in \cite{GS2017} to define bi-Boolean independence. The relevance here is that the partition $\pi_{\omega, \chi}$ was also used in the paper \cite{C2016} to provide another characterization of c-bi-free independence given as follows.

\begin{thm}[\cite{C2016}*{Theorem 8}]
\label{CBFChar}
A family $\{(\A_{k, \ell}, \A_{k, r})\}_{k \in K}$ of pairs of algebras in a double non-commutative space $(\A, \varphi, \psi)$ is c-bi-free with respect to $(\varphi, \psi)$ if and only if whenever $n \geq 1$, $\chi: \{1, \dots, n\} \to \{\ell, r\}$, $\omega: \{1, \dots, n\} \to K$, and $a_1, \dots, a_n \in \A$ with $a_j \in \A_{\omega(j), \chi(j)}$ such that $\psi(a_V) = 0$ for all $V \in \pi_{\omega, \chi}$, it follows that
\[\psi(a_1\cdots a_n) = 0 \qand \varphi(a_1\cdots a_n) = \prod_{V \in \pi_{\chi, \omega}}\varphi(a_V).\]
\end{thm}

Note that the above equations can be used to uniquely determine all mixed moments in terms of pure moments (i.e., moments of the individual pairs of algebras). We desire to prove the associativity of the bi-monotonic product.

\begin{thm}\label{Associativity}
The bi-monotonic product is associative.
\end{thm}

Before presenting the proof of Theorem \ref{Associativity} we note that we may define bi-monotonic independence (of type I) as follows. Note that the definition holds if $K$ is infinite since one need only consider a finite number of $K$ at once when computing moments.

\begin{defn}
A linearly ordered family $\{(\A_{k, \ell}, \A_{k, r})\}_{k \in K}$ of pairs of algebras in a non-commutative space $(\A, \varphi)$ is said to be \textit{bi-monotonically independent} (of type I) with respect to $\varphi$ if the joint distributions of $\{(\A_{k, \ell}, \A_{k, r})\}_{k \in K}$ with respect to $\varphi$ and $\bim_{k \in K}\varphi|_{\A_{k, \ell} \sqcup \A_{k, r}}$ coincide. A linearly ordered two-faced family of non-commutative elements is bi-monotonically independent (of type I) with respect to $\varphi$ if the family of pairs of algebras they generate are bi-monotonically independent (of type I).
\end{defn}

To begin the proof of Theorem \ref{Associativity}, we require the following.

\begin{lem}\label{lem:BiMonoFormula}
Let $(\A_{1, \ell}, \A_{1, r})$ and $(\A_{2, \ell}, \A_{2, r})$ be pairs of algebras in a non-commutative space $(\A, \varphi)$ which are bi-monotonically independent with respect to $\varphi$. If $n \geq 1$, $\chi: \{1, \dots, n\} \to \{\ell, r\}$, $\omega: \{1, \dots, n\} \to \{1, 2\}$, and $a_1, \dots, a_n \in \A$ are such that $a_j \in \A_{\omega(j), \chi(j)}$, then
\begin{equation}\label{BiMonoFormula}
\varphi(a_1\cdots a_n) = \varphi(a_{W})\prod_{\substack{V \in \pi_{\chi, \omega}\\\omega(V) = 2}}\varphi(a_V),
\end{equation}
where $W = \{1, \dots, n\} \setminus \{V\,|\,V \in \pi_{\chi, \omega}, \omega(V) = 2\} = \{j \in \{1,\ldots, n\}\,|\,\omega(j) = 1\}$.
\end{lem}

\begin{proof}
We will use Theorem \ref{CBFChar} to prove this result. There is a simple alternate proof of this result using the additional technology of c-bi-free cumulants from \cite{GS2016}. 

We take the universal free product realizations of bi-monotone independence in Definition \ref{def:bimonotone} and regard the free products as subalgebras of $\A$ via the canonical  homomorphisms $\A_{k,\ell}\sqcup\A_{k,r} \to \A$. We also consider the unitization $\widetilde \A$ of $\A$ into which the unitizations of free products embed. 
Let $\varphi_k = \varphi|_{\A_{k, \ell} \sqcup \A_{k, r}}$ for $k \in \{1, 2\}$, and define $\varphi'$ on $\A' := \widetilde{(\A_{1, \ell} \sqcup \A_{1, r}) \sqcup (\A_{2, \ell} \sqcup \A_{2, r})}$ by $\varphi' = \widetilde{\varphi_1 \bim \varphi_2}$ (so $\varphi'$ applied to the unit is 1). Then $\varphi' = \varphi|_{\A'}$ by the assumption of bi-monotonic independence.
 Thus Theorem \ref{CBFChar} implies for all $n \geq 1$, $\chi: \{1, \dots, n\} \to \{\ell, r\}$, $\omega: \{1, \dots, n\} \to \{1, 2\}$, and $a_1, \dots, a_n \in \widetilde \A$ with $a_j \in \A_{1, \chi(j)}$ if $\omega(j) = 1$ and $a_j \in \widetilde{\A_{2, \chi(j)}}$ if $\omega(j) = 2$ such that $\varphi_2(a_V) = 0$ for $V \in \pi_{\omega, \chi}$ with $\omega(V) = 2$, that
\begin{equation}\label{phi1}
\varphi'(a_1\cdots a_n) = \prod_{V \in \pi_{\chi, \omega}}\varphi_{\omega(V)}(a_V)
=\begin{cases} 
0, & \text{if~} \omega(j)=2 \text{~for some $j$}, \\
\varphi_1(a_1\cdots a_n),& \text{otherwise}, 
\end{cases}
\end{equation}
as $\psi(\A_{1, \ell} \sqcup \A_{1, r}) = \{0\}$ (as $\psi$ is the delta functional on $\A_{1, \ell} \sqcup \A_{1, r}$) and as $\psi(a_V) = \varphi_2(a_V) = 0$ for all $V \in \pi_{\chi, \omega}$ with $\omega(V) = 2$.

On the other hand, define a unital linear functional $\varphi'' : \A' \to \bC$ as follows: for all $n \geq 1$, $\chi: \{1, \dots, n\} \to \{\ell, r\}$, $\omega: \{1, \dots, n\} \to \{1, 2\}$, and $a_1, \dots, a_n \in \widetilde\A$ with $a_j \in \A_{1, \chi(j)}$ if $\omega(j) = 1$ and $a_j \in \widetilde{\A_{2, \chi(j)}}$ if $\omega(j) = 2$, let
\[
\varphi''(a_1\cdots a_n) = \varphi_1(a_{W})\prod_{\substack{V \in \pi_{\chi, \omega}\\\omega(V) = 2}}\varphi_2(a_V)
\]
where $W = \{1, \dots, n\} \setminus \{V\,|\,V \in \pi_{\chi, \omega}, \omega(V) = 2\}$ (and $\varphi''$ applied to the unit is 1). Then it is easy to verify that $\varphi''$ is well-defined.

We claim for all $n \geq 1$, $\chi: \{1, \dots, n\} \to \{\ell, r\}$, $\omega: \{1, \dots, n\} \to \{1, 2\}$, and $a_1, \dots, a_n \in \widetilde\A$ with $a_j \in \A_{1, \chi(j)}$ if $\omega(j) = 1$ and $a_j \in \widetilde{\A_{2, \chi(j)}}$ if $\omega(j) = 2$ that
\[
\varphi''(a_1\cdots a_n) = \varphi'(a_1\cdots a_n) 
\]
which thereby will complete the proof.  Notice if $\omega$ is constant, then the claim holds as $\varphi'' = \varphi_{\omega(1)} = \varphi'$ in this case. Therefore we may assume that $\omega$ is not constant. 

To complete the claim, we proceed by induction on $n$ with the base case $n = 1$ following from the constant $\omega$ case.  Suppose the result holds for $n-1$ for some $n \geq 2$.  Fix $\chi: \{1, \dots, n\} \to \{\ell, r\}$, $\omega: \{1, \dots, n\} \to \{1, 2\}$, and $a_1, \dots, a_n \in \widetilde\A$ with $a_j \in \A_{1, \chi(j)}$ if $\omega(j) = 1$ and $a_j \in \widetilde{\A_{2, \chi(j)}}$ if $\omega(j) = 2$. For each block $V = \{v_1 < \cdots < v_s\}$ of $\pi_{\omega, \chi}$, let $\lambda_V$ be a complex root of the polynomial 
\[
\varphi_2((a_{v_1} - z)\cdots (a_{v_s} - z)) = \varphi'((a_{v_1} - z)\cdots (a_{v_s} - z)) = \varphi''((a_{v_1} - z)\cdots (a_{v_s} - z)).
\]
For each $j \in \{1, \dots, n\}$, define $a_j^\circ$ as follows: if $V$ is the block of $\pi_{\omega, \chi}$ containing $j$ then $a_j^\circ = a_j$ if $\omega(V) = 1$ and $a^\circ_j = a_j - \lambda_V$ (in the unitization) if $\omega(V) = 2$.  Therefore we obtain that
\[
\varphi''(a_1^\circ \cdots a^\circ_n) = 0 = \varphi'(a^\circ_1\cdots a^\circ_n).
\]
Hence, as
\[
\varphi'(a_1^\circ \cdots a^\circ_n) = \varphi'(a_1\cdots a_n) + L \qand \varphi''(a_1^\circ \cdots a^\circ_n) = \varphi''(a_1\cdots a_n) + L
\]
where $L$ are lower-order terms that are equal by the inductive hypothesis, the result follows.
\end{proof}

Using Lemma \ref{lem:BiMonoFormula}, it is fairly straightforward to verify the associativity of (type I) bi-monotonic product.

\begin{proof}[Proof of Theorem \ref{Associativity}]
Let $(\A_{1, \ell}, \A_{1, r})$, $(\A_{2, \ell}, \A_{2, r})$, and $(\A_{3, \ell}, \A_{3, r})$ be three pairs of algebras with linear functionals $\varphi_k: \A_{k, \ell} \sqcup \A_{k, r} \to \bC$. Denote
\[
\A := ((\A_{1, \ell} \sqcup \A_{1, r}) \sqcup (\A_{2, \ell} \sqcup \A_{2, r})) \sqcup (\A_{3, \ell} \sqcup \A_{3, r}) \simeq (\A_{1, \ell} \sqcup \A_{1, r}) \sqcup ((\A_{2, \ell} \sqcup \A_{2, r}) \sqcup (\A_{3, \ell} \sqcup \A_{3, r}))
\]
under the natural identification, and let $\varphi'$ and $\varphi''$ be the linear functionals on $\A$ defined by
\[
\varphi' = (\varphi_1 \bim \varphi_2) \bim \varphi_3 \qand \varphi'' = \varphi_1 \bim (\varphi_2 \bim \varphi_3).
\]
It suffices to show that $\varphi' = \varphi''$.

To see this, let $n \geq 1$, $\chi: \{1, \dots, n\} \to \{\ell, r\}$, $\omega: \{1, \dots, n\} \to \{1, 2, 3\}$, and $a_1, \dots, a_n \in \A$ with $a_j \in \A_{\omega(j), \chi(j)}$ arbitrary.   Consider $\pi_{\chi, \omega}$ and let $V_1, \dots, V_m$ denote the ordered blocks $\pi_{\chi, \omega}$ such that each $V_k$ is a $\chi$-interval, $\max_{\prec_\chi}(V_k) \prec_\chi \min_{\prec_\chi}(V_{k + 1})$, $\omega$ is constant on each $V_k$, and $\omega(V_k) \neq \omega(V_{k + 1})$ for all $1 \leq k \leq m - 1$.  Define two $\chi$-intervals $V_j, V_k$ with $j < k$ and $\omega(V_j) = \omega(V_k) = 2$ to be equivalent, denoted $V_j \sim V_k$, if $\omega(x) \in \{2, 3\}$ for all $x$ such that
\[
\max_{\prec_\chi}(V_j) \prec_\chi x \prec_\chi \min_{\prec_\chi}(V_{k}).
\]
Let $Z$ denote a set of representatives from each equivalence classes under $\sim$ and let
\[
Y = \left\{\left. \bigcup_{V_k \sim V} V_k \, \right| \, V \in Z\right\}.
\]
Furthermore let
\[
W = \{j \in \{1, \ldots, n\} \, \mid \, \omega(j) = 1\}.
\] 
We claim that
\[
\varphi'(a_1 \cdots a_n) =  \varphi(a_W) \prod_{V' \in Y} \varphi(a_{V'}) \prod_{\substack{V \in \pi_{\chi, \omega} \\ \omega(V) = 3}} \varphi(a_V)  = \varphi''(a_1 \cdots a_n)
\]
(where empty products are 1 and the functional applied to an empty sequence is also 1).

To see this, first notice the definition of $\varphi'$ and Lemma \ref{lem:BiMonoFormula} imply that
\begin{align*}
\varphi'(a_1 \cdots a_n) &= \varphi(a_{W'}) \prod_{\substack{V \in \pi_{\chi, \omega} \\ \omega(V) = 3}} \varphi(a_V)
\end{align*}
where $W' = \{k \in \{1,\ldots, n\} \, \mid \, \omega(k) \neq 3\}$.  However, considering $\pi_{\chi|_{W'}, \omega|_{W'}}$, it is elementary to see that if $V'_1, \dots, V'_{m'}$ are the ordered blocks of $\pi_{\chi|_{W'}, \omega|_{W'}}$ such that each $V'_k$ is a $\chi|_{W'}$-interval, $\max_{\prec_\chi}(V'_k) \prec_\chi \min_{\prec_\chi}(V'_{k + 1})$, $\omega$ is constant on each $V'_k$, and $\omega(V'_k) \neq \omega(V'_{k + 1})$ for all $1 \leq k \leq m' - 1$, then
\[
\{V'_k \, \mid \, \omega(V'_k) = 2\} = Y.
\]
Hence
\[
\varphi(a_{W'})  = \varphi(a_W) \prod_{V' \in Y} \varphi(a_{V'})
\]
so the desired formula holds for $\varphi'(a_1 \cdots a_n)$.

On the other hand, consider $\omega'' : \{1,\ldots, n\} \to \{1,4\}$, where 
\[
\omega''(k) = \begin{cases}
1 & \text{if } \omega(k) = 1 \\
4 & \text{if } \omega(k) \in \{2,3\}
\end{cases}.
\]
Hence the definition of $\varphi''$ and Lemma \ref{lem:BiMonoFormula} imply that
\begin{align*}
\varphi''(a_1 \cdots a_n) &= \varphi(a_{W}) \prod_{\substack{V \in \pi_{\chi, \omega''} \\ \omega(V) = 4}} \varphi(a_V).
\end{align*}
However, it is not difficult to see that each $V \in \pi_{\chi, \omega''}$ is a union of $V_k$ such that $\omega(V_k) = 3$ and an element from $Y$.  Hence
\[
\prod_{\substack{V \in \pi_{\chi, \omega''} \\ \omega(V) = 4}} \varphi(a_V) = \prod_{V' \in Y} \varphi(a_{V'}) \prod_{\substack{V \in \pi_{\chi, \omega} \\ \omega(V) = 3}}\varphi(a_V)
\]
as desired.
\end{proof}

As a direct corollary of the associativity, we obtain a characterization of (type I) bi-monotonic independence of the same flavour as that of monotonic independence.
\begin{cor}
\label{cor:bi-monotone-type-I-looks-like-monotone}
Let $K$ be a set equipped with a linear order $<$ and let $\{(\A_{k, \ell}, \A_{k, r})\}_{k \in K}$  be pairs of algebras in a non-commutative space $(\A, \varphi)$ which are bi-monotonically independent with respect to $\varphi$.   Let $n \geq 1$, $\chi: \{1, \dots, n\} \to \{\ell, r\}$, $\omega: \{1, \dots, n\} \to K$, and $a_1, \dots, a_n \in \A$ be such that $a_j \in \A_{\omega(j), \chi(j)}$.  Let $V_1, \dots, V_m$ denote the ordered blocks $\pi_{\chi, \omega}$ such that each $V_k$ is a $\chi$-interval, $\max_{\prec_\chi}(V_k) \prec_\chi \min_{\prec_\chi}(V_{k + 1})$, $\omega$ is constant on each $V_k$, and $\omega(V_k) \neq \omega(V_{k + 1})$ for all $1 \leq k \leq m - 1$.   If $k$ is such that $\omega(V_{k-1}) < \omega(V_k) > \omega(V_{k+1})$ (where one inequality is eliminated if $k = 1$ or $k = n$), then
\[
\varphi(a_1 \cdots a_n) = \varphi(a_{V_k}) \varphi(a_W)
\]
where $W = \{1,\ldots, n\} \setminus V_k$.
\end{cor}
\begin{proof} 
Let $\B_{1,\ep}$ and $\B_{2,\ep}$ be the subalgebras of $\A$ generated by $\{\A_{i,\ep}: i < \omega(V_k)\}$ and $\{\A_{i,\ep}: i \geq \omega(V_k)\}$ for $\ep\in\{\ell,r\}$, respectively. The associativity of bi-monotonic product implies that the pairs $(\B_{1,\ell},\B_{1,r})$ and $(\B_{2,\ell},\B_{2,r})$ are bi-monotonically independent in this order, and then Lemma \ref{lem:BiMonoFormula} implies our desired formula. 
\end{proof}

\begin{exam}[Bi-monotonic product does not preserve states] \label{not-state} Unfortunately, the bi-monotonic product of states need not be a state.  To see this, begin by supposing that $a_1,a_2,a_3$ are elements of a non-commutative space $(\A,\varphi)$ satisfying bi-monotonic independence prescribed by $(\omega(1),\omega(2),\omega(3)) = (2, 1, 2)$ and $(\chi(1),\chi(2),\chi(3)) = (\ell, \ell, r)$ under the notation of Corollary \ref{cor:bi-monotone-type-I-looks-like-monotone}.  Then $V_1 = \{1\}, V_2 = \{2\},V_3 = \{3\},$ and so 
$$
\varphi(a_1 a_2 a_3) = \varphi(a_1) \varphi(a_2) \varphi(a_3).
$$
In the case that $\A$ is a unital $^*$-algebra, $\varphi$ is a state (or self-adjoint) and $a_1, a_2, a_3$ are self-adjoint, the right-hand-side of the above equation is real so that
$$
\varphi(a_1 a_2 a_3) = \overline{\varphi(a_1 a_2 a_3)} = \varphi( (a_1 a_2 a_3)^* ) = \varphi(a_3 a_2 a_1).
$$
Defining $b_1=a_3, b_2=a_2$ and $b_3=a_1$ then shows that $(b_1,b_2,b_3)$ satisfies the bi-monotonic independence prescribed by $(\chi'(1),\chi'(2),\chi'(3))=(r,\ell,\ell)$ and $(\omega'(1),\omega'(2),\omega'(3)) = (2,1,2)$, and hence $V_1' = \{2\},
V_2' = \{1,3\}$. Then we have 
$$
\varphi(a_3 a_2 a_1) = \varphi(a_2) \varphi(a_3 a_1).
$$
Combining the above formulas, if $\varphi(a_2)\neq 0$ then we obtain that
$$
\varphi(a_3 a_1) = \varphi(a_1)\varphi(a_3), 
$$
which is also equal to $\varphi(a_1 a_3)$ by taking the complex conjugate. 
The above arguments show that, if $(\A_{1,\ell},\A_{1,r})$ and $(\A_{2,\ell},\A_{2,r})$ are bi-monotonically independent (in this order) pairs of $^*$-subalgebras in a non-commutative probability space $(\A,\varphi)$, then the identity $\varphi(a b) = \varphi( b a)=\varphi(a)\varphi(b)$ holds for all $a \in \A_{2,\ell}$ and $b\in \A_{2,r}$ unless $\varphi|_{\A_{1,\ell}}=0$. Since our definition of bi-monotonic product is universal, the bi-monotonic product does not preserve states in general. 
\end{exam}

The above example does not immediately preclude the possibility of a bi-monotonic convolution for measures.  Indeed recall that a commuting pair of self-adjoint elements $(a,b)$ in a C$^*$-non-commutative probability space $(\A,\varphi)$ has a compactly supported probability distribution $\mu$ on $\bR^2$, namely
$$
\varphi(a^m b^n) = \int_{\bR^2} s^m t^n\,d\mu(s,t), \qquad m,n \in \bN\cup\{0\}. 
$$
Thus the question of whether there is a bi-monotonic convolution for measures reduces to the question of whether $(a_1+a_2,b_1+b_2)$ has a probability distribution with respect to the bi-monotonic product of the states corresponding to $(a_1, b_1)$ and $(a_2,b_2)$.  The techniques of Example \ref{not-state} break down as it is possible to verify that the bi-monotonic product of the states corresponding to $(a_1, b_1)$ and $(a_2,b_2)$ is self-adjoint on the $*$-algebra generated by $(a_1+a_2,b_1+b_2)$.  The question of the existence of a bi-monotonic convolution for measures is answered in the negative in Section \ref{sec:analytic}.

To conclude this subsection, we note several reasons why type I bi-monotonic independence appears to be the more desired bi-probability analogue of monotonic independence.  First, Corollary \ref{cor:bi-monotone-type-I-looks-like-monotone} looks exactly like the definition of monotonic independence given in the introduction modulo a permutation.  This is consistent with the theories of bi-free independence and bi-Boolean independence.  Furthermore, as in these other bi-probability theories, Corollary \ref{cor:bi-monotone-type-I-looks-like-monotone} directly implies if $\{(\A_{k, \ell}, \A_{k, r})\}_{k \in K}$ are type I bi-monotonically independent, then the left algebras $\{\A_{k, \ell}\}_{k \in K}$ are monotonically independent, the right algebras $\{\A_{k, r}\}_{k \in K}$ are monotonically independent, and $A_{k, \ell}$ and $A_{j, r}$ are classically independent (that is, they commute in distribution and the moment of a product is the product of the moments) whenever $j \neq k$.

\subsection{The additive bi-monotonic convolution}\label{sec:convolution}

We conclude this section by returning to the Equation \eqref{eqn:biMonCauchy} produced by operator-valued monotonic independence. Since we do not have positivity we will treat transforms as formal power series. For example, for a pair $(a,b)$ in a non-commutative space $(\A,\varphi)$ we define 
$$
G_{(a,b)}(z,w) = \sum_{m,n=0}^\infty \frac{\varphi(a^m b^n)}{z^{m+1}w^{n+1}}
$$  
as formal power series (but of course corresponds to an analytic function away from $(0,0)$ when $\varphi$ is a state on the C$^*$-algebra generated by $(a,b)$).

Recall as shown in \cite{GS2016}*{Corollary 5.7} that for a pair $(a, b)$ in a double non-commutative space $(\A, \varphi, \psi)$, the reduced c-bi-free partial $R$-transform of $(a, b)$ is given by 
\[\widetilde{R}^{\mathrm{c}}_{(a, b)}(z, w) = \frac{zwG^\varphi_{(a, b)}(K_a^\psi(z), K_b^\psi(w))}{G_a^\varphi(K_a^\psi(z))G_b^\varphi(K_b^\psi(w))G_{(a, b)}^\psi(K_a^\psi(z), K_b^\psi(w))} - \frac{zw}{G_{(a, b)}^\psi(K_a^\psi(z), K_b^\psi(w))}
\]
as formal power series (by which we mean multiplying both sides by a common denominator), where the superscript $\varphi$ or $\psi$ indicates the linear functional with respect to which the corresponding transform is defined. The function $\widetilde{R}^{\mathrm{c}}_{(a, b)}$ linearizes the additive c-bi-free convolution in the sense that
\[\widetilde{R}^{\mathrm{c}}_{(a_1 + a_2, b_1 + b_2)}(z, w) = \widetilde{R}^{\mathrm{c}}_{(a_1, b_1)}(z, w) + \widetilde{R}^{\mathrm{c}}_{(a_2, b_2)}(z, w)\]
if $(a_1, b_1)$ and $(a_2, b_2)$ are c-bi-free with respect to $(\varphi, \psi)$.

Let $(a_1, b_1)$ and $(a_2, b_2)$ be pairs of elements in a non-commutative space $(\A, \varphi)$ which are bi-monotonically independent with respect to $\varphi$.  Let $\varphi_k = \varphi|_{\alg(a_k, b_k)}$ for $k \in \{1, 2\}$. Since the bi-monotonic product $\varphi_1 \bim \varphi_2$ is defined to be the restriction of the first coordinate of the c-bi-free product $(\widetilde{\varphi_1}, \delta_1) *\!* (\widetilde{\varphi_2}, \widetilde{\varphi_2})$, and since the transform $\widetilde{R}^{\mathrm{c}}_{(a_1, b_1)}$ with respect to $(\widetilde{\varphi_1}, \delta_1)$ is given by
\[
\widetilde{R}^{\mathrm{c}}_{(a_1, b_1)}(z, w) = \frac{G_{(a_1, b_1)}^\varphi(1/z, 1/w)}{G_{a_1}^\varphi(1/z)G_{b_1}^\varphi(1/w)} - 1,
\]
the transform $\widetilde{R}^{\mathrm{c}}_{(a_2, b_2)}$ with respect to $(\widetilde{\varphi_2}, \widetilde{\varphi_2})$ is given by
\[
\widetilde{R}^{\mathrm{c}}_{(a_2, b_2)}(z, w) = 1 - \frac{zw}{G_{(a_2, b_2)}^\varphi(K_{a_2}^\varphi(z), K_{b_2}^\varphi(w))},
\]
and the transform $\widetilde{R}^{\mathrm{c}}_{(a_1 + a_2, b_1 + b_2)}$ with respect to $(\widetilde{\varphi_1}, \delta_1) *\!* (\widetilde{\varphi_2}, \widetilde{\varphi_2})$ is given by
\[
\widetilde{R}^{\mathrm{c}}_{(a_1 + a_2, b_1 + b_2)}(z, w) = \frac{zwG^\varphi_{(a_1 + a_2, b_1 + b_2)}(K_{a_2}^\varphi(z), K_{b_2}^\varphi(w))}{G_{a_1 + a_2}^\varphi(K_{a_2}^\varphi(z))G_{b_1 + b_2}^\varphi(K_{b_2}^\varphi(w))G^\varphi_{(a_2, b_2)}(K_{a_2}^\varphi(z), K_{b_2}^\varphi(w))} - \frac{zw}{G_{(a_2, b_2)}^\varphi(K_{a_2}^\varphi(z), K_{b_2}^\varphi(w))},
\]
(as $K^\psi_{a_1 + a_2}(z) = K^{\delta_1}_{a_1}(z) + K^\varphi_{a_2}(z) = K^\varphi_{a_2}(z)$ and similarly $K^\psi_{b_1 + b_2}(w) = K^\varphi_{b_2}(w)$), we have that
\[\frac{zwG^\varphi_{(a_1 + a_2, b_1 + b_2)}(K_{a_2}^\varphi(z), K_{b_2}^\varphi(w))}{G_{a_1 + a_2}^\varphi(K_{a_2}^\varphi(z))G_{b_1 + b_2}^\varphi(K_{b_2}^\varphi(w))G^\varphi_{(a_2, b_2)}(K_{a_2}^\varphi(z), K_{b_2}^\varphi(w))} = \frac{G_{(a_1, b_1)}^\varphi(1/z, 1/w)}{G_{a_1}^\varphi(1/z)G_{b_1}^\varphi(1/w)}\]
by the additivity of the reduced c-bi-free $R$-transform. Since
\[\frac{1}{G^\varphi_{a_1 + a_2}(K^\varphi_{a_2}(z))} = F_{a_1 + a_2}^\varphi(K^\varphi_{a_2}(z)) = F^\varphi_{a_1}(F^\varphi_{a_2}(K^\varphi_{a_2}(z))) = \frac{1}{G^\varphi_{a_1}(1/z)},\]
and similarly $1/G^\varphi_{b_1 + b_2}(K^\varphi_{b_2}(w)) = 1/G^\varphi_{b_1}(1/w)$, we have
\[G^\varphi_{(a_1 + a_2, b_1 + b_2)}(K_{a_2}^\varphi(z), K_{b_2}^\varphi(w)) = \frac{1}{zw}G_{(a_1, b_1)}^\varphi(1/z, 1/w)G^\varphi_{(a_2, b_2)}(K_{a_2}^\varphi(z), K_{b_2}^\varphi(w)).\]
Replacing $z$ by $G_{a_2}^\varphi(z)$ and $w$ by $G_{b_2}^\varphi(w)$ produces
\[G^\varphi_{(a_1 + a_2, b_1 + b_2)}(z, w) = G_{(a_1, b_1)}^\varphi(F^\varphi_{a_2}(z), F^\varphi_{b_2}(w))G^\varphi_{(a_2, b_2)}(z, w)F^\varphi_{a_2}(z)F^\varphi_{b_2}(w),\]
which is exactly Equation \eqref{eqn:biMonCauchy} as predicted.

\subsection{Bi-monotonic independence of type II}

In \cite{M2000}*{Section 2}, Muraki presented a monotonic product which naturally realizes monotonic independence. While studying Muraki's construction, we noticed that there is also a choice of a left and a right representation of operators of the initial spaces on the monotonic product space like Voiculescu's bi-free construction in \cite{V2014}*{Section 1}. Thus a notion of bi-monotonic independence (of type II) arises when considering the left and right representations simultaneously, which we discuss as follows, starting with Muraki's monotonic product construction.  We note this was independently discovered and studied in \cite{G2017}.

Let $\{(\X_k, \X_k^\circ, \xi_k)\}_{k \in K}$ be a linearly ordered family where, for each $k \in K$, $\X_k$ is a vector space, $\X_k^\circ \subset \X_k$ is a subspace of co-dimension $1$, and $0 \neq \xi_k \in \X_k$ is a vector such that $\X_k = \bC\xi_k \oplus \X_k^\circ$. For each triple $(\X_k, \X_k^\circ, \xi_k)$, define a linear functional $\varphi_k: \X_k \to \bC$ such that $\varphi_k(x)\xi_k = \mathfrak{p}_k(x)$ where $\mathfrak{p}_k$ is the projection of $\X$ onto $\bC \xi_k$ with respect to the direct sum   $\X_k = \bC\xi_k \oplus \X_k^\circ$.

Given a vector space $\X$, let $\L(\X)$ denote the algebra of linear operators on $\X$. Given $(\X, \X^\circ, \xi)$ as above, there is a linear functional $\varphi_\xi: \L(\X) \to \bC$ defined by $\varphi_\xi(T) = \varphi(T\xi)$ for all $T \in \L(\X)$. Since $\varphi_\xi(I) = \varphi(I\xi) = \varphi(\xi) = 1$ for the identity operator $I$ on $\X$, the pair $(\L(\X), \varphi_\xi)$ is a non-commutative space.

Given a linearly ordered family $\{(\X_k, \X_k^\circ, \xi_k)\}_{k \in K}$ of vector spaces with specified vectors, the \textit{monotonic product} $(\X, \X^\circ, \xi) = \rhd_{k \in K}(\X_k, \X_k^\circ, \xi_k)$ is the vector space
\[
\X = \bC\xi \oplus \X^\circ\quad\text{with}\quad\X^\circ = \bigoplus_{n \geq 1}\left(\bigoplus_{k_1 > \cdots > k_n}\X_{k_1}^\circ \otimes \cdots \otimes \X_{k_n}^\circ\right).
\]

For every $k \in K$ and $\sharp \in \{=, <, >\}$ let
\begin{align*}
\X^\circ_{(\ell, \sharp k)} &= \bigoplus_{n \geq 1}\left(\bigoplus_{\substack{k_1 > \cdots > k_n\\k_1 \sharp k}}\X_{k_1}^\circ \otimes \cdots \otimes \X_{k_n}^\circ\right) \qqand 
\X^\circ_{(r, \sharp k)} = \bigoplus_{n \geq 1}\left(\bigoplus_{\substack{k_1 > \cdots > k_n\\k_n \sharp k}}\X_{k_1}^\circ \otimes \cdots \otimes \X_{k_n}^\circ\right).
\end{align*}
Writing
\[
\X_k = \bC\xi_k \oplus \X_k^\circ \qqand
\X_{(\ell, < k)}  = \bC\xi_{(\ell, < k)} \oplus \X_{(\ell, < k)}^\circ
\]
and using the decomposition
\[
\X = \bC\xi \oplus \X^\circ_{(\ell, = k)} \oplus \X^\circ_{(\ell, < k)} \oplus \X^\circ_{(\ell, > k)} = \bC\xi \oplus \X_k^\circ \oplus (\X_k^\circ \otimes \X^\circ_{(\ell, < k)}) \oplus \X^\circ_{(\ell, < k)} \oplus \X^\circ_{(\ell, > k)},
\]
it was observed in \cite{M2000} that there are natural identifications
\[
V_k: (\X_k \otimes \X_{(\ell, < k)}) \oplus \X^\circ_{(\ell, > k)} \to \X
\]
given by
\[\xi_k \otimes \xi_{(\ell, < k)} \to \xi,\quad\X_k^\circ \otimes \xi_{(\ell, < k)} \to \X_k^\circ,\quad\xi_k \otimes \X_{(\ell, < k)}^\circ \to \X_{(\ell, < k)}^\circ.\]
Similarly, writing
\[
\X_{(r, > k)} = \bC\xi_{(r, > k)} \oplus \X_{(r, > k)}^\circ \qqand
\X_k  = \bC\xi_k \oplus \X_k^\circ 
\]
and using the decomposition
\[
\X = \bC\xi \oplus \X^\circ_{(r, = k)} \oplus \X^\circ_{(r, < k)} \oplus \X^\circ_{(r, > k)} = \bC\xi \oplus \X_k^\circ \oplus (\X^\circ_{(r, > k)} \otimes \X_k^\circ) \oplus \X^\circ_{(r, > k)} \oplus \X^\circ_{(r, < k)},\]
there are natural identifications
\[W_k: (\X_{(r, > k)} \otimes \X_k) \oplus \X_{(r, < k)}^\circ \to \X\]
given by
\[\xi_{(r, > k)} \otimes \xi_k \to \xi,\quad\xi_{(r, > k)} \otimes \X_k^\circ \to \X_k^\circ,\quad\X_{(r, > k)}^\circ \otimes \xi_k \to \X_{(r, > k)}^\circ.\]
Consequently, there are natural left and right representations
\[\lambda_k: \L(\X_k) \to \L(\X)\qand\rho_k: \L(\X_k) \to \L(\X)\]
by
\[\lambda_k(T) = V_k\left((T \otimes 1_{\X_{(\ell, < k)}}) \oplus 0_{\X^\circ_{(\ell, > k)}}\right)V_k^{-1}\qand\rho_k(T) = W_k\left((1_{\X_{(r, > k)}} \otimes T) \oplus 0_{\X_{(r, < k)}^\circ}\right)W_k^{-1}\]
for every $T \in \L(\X_k)$.

Notice for $k \in K$ and $T \in \L(\X_k)$, that
\[\lambda_k(T)\xi = \rho_k(T)\xi = \varphi_{\xi_k}(T)\xi \oplus (1 - \mathfrak{p}_k)T\xi_k,\]
so both $\lambda_k$ and $\rho_k$ satisfy $\varphi\circ \lambda_k = \varphi_{\xi_k}$ and $\varphi\circ \rho_k = \varphi_{\xi_k}$. Moreover, for $x_1 \otimes \cdots \otimes x_n \in \X_{k_1}^\circ \otimes \cdots \otimes \X_{k_n}^\circ$ with $k_1 > \cdots > k_n$, we have
\[
\lambda_k(T)(x_1 \otimes \cdots \otimes x_n) = \begin{cases}
0 &\text{if }\,k < k_1\\
(\varphi_{k}(Tx_1)x_2 \otimes \cdots \otimes x_n) \oplus ((1 - \mathfrak{p}_k)Tx_1 \otimes x_2 \otimes \cdots \otimes x_n) &\text{if }\,k = k_1\\
(\varphi_{\xi_k}(T)x_1 \otimes \cdots \otimes x_n) \oplus ((1 - \mathfrak{p}_k)T\xi_k \otimes x_1 \otimes \cdots \otimes x_n) &\text{if }\,k > k_1
\end{cases},
\]
and similarly
\[
\rho_k(T)(x_1 \otimes \cdots \otimes x_n) = \begin{cases}
(\varphi_{\xi_k}(T)x_1 \otimes \cdots \otimes x_n) \oplus (x_1 \otimes \cdots \otimes x_n \otimes (1 - \mathfrak{p}_k)T\xi_k) &\text{if }\,k < k_n\\
(\varphi_k(Tx_n)x_1 \otimes \cdots \otimes x_{n - 1}) \oplus (x_1 \otimes \cdots \otimes x_{n - 1} \otimes (1 - \mathfrak{p}_k)Tx_n) &\text{if }\,k = k_n\\
0 &\text{if }\,k > k_n
\end{cases},
\]
where the empty tensor product is $\xi$.

In view of the above construction, bi-monotonic independence (of type II) is defined as follows.

\begin{defn}
A linearly ordered family $\{(\A_{k, \ell}, \A_{k, r})\}_{k \in K}$ of pairs of algebras in a non-commutative space $(\A, \varphi)$ is said to be \textit{bi-monotonically independent} (of type II) with respect to $\varphi$ if there are a linearly ordered family $\{(\X_k, \X_k^\circ, \xi_k)\}_{k \in K}$ of vector spaces with specified vectors and homomorphisms
\[\ell_k: \A_{k, \ell} \to \L(\X_k) \qand r_k: \A_{k, r} \to \L(\X_k)\]
such that the joint distribution of $\{(\A_{k, \ell}, \A_{k, r})\}_{k \in K}$ with respect to $\varphi$ is equal to the joint distribution of the linearly ordered family
\[\{(\lambda_k \circ \ell_k(\A_{k, \ell}), \rho_k \circ r_k(\A_{k, r}))\}_{k \in K}\]
of pairs of algebras in $(\L(\X), \varphi_\xi)$ with respect to $\varphi_\xi$, where $(\X, \X^\circ, \xi) = \rhd_{k \in K}(\X_k, \X_k^\circ, \xi_k)$.
\end{defn}

Note that bi-monotonic independence (of type II) of $\{(\A_{k, \ell}, \A_{k, r})\}_{k \in K}$ implies monotonic independence of the left algebras $\{\A_{k, \ell}\}_{k \in K}$ and anti-monotonic independence of the right algebras $\{\A_{k, r}\}_{k \in K}$, where the first claim was shown in \cite{M2000}*{Theorem 2.1} and the second claim can be shown via a routine check similar to the proof of \cite{M2000}*{Theorem 2.1}.  Hence bi-monotonic independence of type II is different from bi-monotonic independence of type I as $\{\A_{k, r}\}_{k \in K}$ are monotonically independent in the latter.

On the other hand, unlike other independences for pairs of algebras (e.g., bi-free, bi-Boolean, or c-bi-free independences), if $\{(\A_{k, \ell}, \A_{k, r})\}_{k \in K}$ is bi-monotonically independent (of type II) with respect to $\varphi$, then a left algebra $\A_{k, \ell}$ and a right algebra $\A_{j, r}$ are in general not classically independent with respect to $\varphi$, unless $k > j$. Indeed, if $\{(\A_{k, \ell}, \A_{k, r})\}_{k \in K}$ are bi-monotonically independent of type II, $k,j \in K$ are such that $k < j$, $a \in \A_{k, \ell}$, and $b \in \A_{j, r}$ then it is easy via the above representations to verify that
\[
\varphi(abab) = \varphi(a)^2 \varphi(b)^2 \neq \varphi(a^2) \varphi(b^2) = \varphi(a^2b^2).
\]
Consequently, the type II notion does not lead to an additive convolution on probability measures on $\bR^2$ nor on finite signed measures on $\bR^2$.  Indeed the bi-monotonic convolution of type II of the pairs $(a, 0)$ and $(0, b)$ would be $(a,b)$ whose distribution cannot correspond to any form of measure on $\bR^2$ as $[a,b] \neq 0$ in distribution.

As previously mentioned, in Section \ref{sec:analytic} it will also be demonstrated that the bi-monotonic convolution of type I of probability measures need not be a probability measure.  It is unknown whether the same is true for finite signed measures. 

Throughout the rest, by bi-monotonic independence we refer to the type I sense.

\section{Bi-monotonic cumulants}\label{sec:cumulants}

In non-commutative probability theory, cumulants play an important role due to the vanishing characterization of the corresponding independence. In particular, given a non-commutative probability space $(\A, \varphi)$, the free cumulants are a family of multilinear functionals $\{\kappa_n: \A^n \to \bC\}_{n \geq 1}$ uniquely determined by the free moment-cumulant formula (see \cite{S1994} for details) with the property that a family $\{\A_k\}_{k \in K}$ of unital subalgebras of $\A$ is free with respect to $\varphi$ if and only if
\[
\kappa_n(a_1, \dots, a_n) = 0
\]
whenever $n \geq 2$, $a_j \in \A_{k_j}$, $k_j \in K$, and there exist $i$ and $j$ such that $k_i \neq k_j$.

In the pairs of algebras setting, Voiculescu demonstrated in \cite{V2014}*{Section 5} the existence of bi-free cumulants $\{\kappa_\chi:\A^n \to \bC\}_{n \geq 1, \chi: \{1, \dots, n\} \to \{\ell, r\}}$ which play the same role as free cumulants when it comes to bi-free independence. The explicit formula was conjectured in \cite{MN2015} and proved in \cite{CNS2015-1}.  In particular, a family $\{(\A_{k, \ell}, \A_{k, r})\}_{k \in K}$ of pairs of unital subalgebras of $\A$ is bi-free with respect to $\varphi$ if and only if for all $n \geq 2$, $\chi: \{1, \dots, n\} \to \{\ell, r\}$, $\omega: \{1, \dots, n\} \to K$, and $a_1, \dots, a_n \in \A$ with $a_j \in \A_{\omega(j), \chi(j)}$, we have that
\[
\kappa_\chi(a_1, \dots, a_n) = 0
\]
whenever $\omega$ is not constant. Note that although $\kappa_\chi(a_1, \dots, a_n)$ is defined for all $a_1, \dots, a_n \in \A$, when used to characterize bi-free independence it is assumed that the $j^{\mathrm{th}}$ argument comes from a left or a right algebra $\A_{\chi(j), \omega(j)}$ depending on whether $\chi(j) = \ell$ or $\chi(j) = r$. Consequently the assumption below that the $j^{\mathrm{th}}$ entry in our bi-monotonic cumulants hails from an algebra dictated by $\chi(j)$ is inconsequential.

Since monotonic independence is non-symmetric one cannot expect the existence of a family of cumulants with the vanishing characterization. However, using the associativity of monotonic independence and the dot operation, the monotonic cumulants were defined in \cite{HS2011-1} for a single random variable and extended in \cite{HS2011-2} to the multivariate case. Moreover, a general (monotonic) moment-cumulant formula was proved in \cite{HS2011-2}*{Theorem 5.3} in perfect analogy with other moment-cumulant formulae. We shall use a similar approach as in \cites{HS2011-1, HS2011-2} to define the bi-monotonic cumulants.

\subsection{The dot operation}

A crucial ingredient in defining the monotonic cumulants is the dot operation introduced in \cites{HS2011-1, HS2011-2}. Roughly speaking, if $a_1, \dots, a_n$ are random variables in a non-commutative probability space $(\A, \varphi)$, then for $N \geq 1$, $(N.a_1, \dots, N.a_n)$ denotes the tuple $(a_1^{(1)} + \cdots + a_1^{(N)}, \dots, a_n^{(1)} + \cdots + a_n^{(N)})$, where $\{a_1^{(i)}, \dots, a_n^{(i)}\}_{i  = 1}^N$ are identically distributed and monotonically independent with respect to $\varphi$. This can always be achieved by enlarging $(\A, \varphi)$ and using the monotonic product. Then the monotonic cumulants satisfy $K_n(N.a_1, \dots, N.a_n) = NK_n(a_1, \dots, a_n)$. In the pairs of algebras setting, we introduce a dot operation as follows.

Let $(\A, \varphi)$ be a non-commutative space and let $\A_\ell$ and $\A_r$ be subalgebras of $\A$. Take copies $\A^{(i)} = \A_{\ell} \sqcup \A_r$ and $\varphi^{(i)} = \varphi|_{\A^{(i)}}$ for $ i \geq1$, and let $\widetilde{\A} = \sqcup_{i\geq1}\A^{(i)}$ and $\widetilde{\varphi} = \bim_{ i \geq1}\varphi^{(i)}$. Then for $n \geq 1$, $\chi: \{1, \dots, n\} \to \{\ell, r\}$, $a_j \in \A_{\chi(j)}$ and for $N \geq1$, define $(N.a_1, \dots, N.a_n)$ to be the tuple $(a_1^{(1)} + \cdots + a_1^{(N)}, \dots, a_n^{(1)} + \cdots + a_n^{(N)})$, where $a_j^{(i)}$ denotes the element in $\A^{(i)}$ corresponding to $a_j$. By construction, for every $\chi: \{1, \dots, n\} \to \{\ell, r\}$, the two-faced families $\{(\{a_p^{(i)}\}_{p \in \chi^{-1}(\{\ell\})}, \{a_q^{(i)}\}_{q \in \chi^{-1}(\{r\})})\}_{i = 1}^\infty$ are identically distributed and bi-monotonically independent with respect to $\widetilde{\varphi}$. As with the monotonic case (see \cite{HS2011-2}*{Proposition 2.4}), this dot operation can be iterated more than once and  
\[\widetilde{\varphi}(M.(N.a_1)\cdots M.(N.a_n)) = \widetilde{\varphi}((MN).a_1\cdots(MN).a_n)\]
for $M, N \geq 1$ since the bi-monotonic product is associative. For simplicity, we also denote the functional $\widetilde \varphi$ by $\varphi$. 

Using the dot operation, the bi-monotonic cumulants are defined as follows.

\begin{defn}\label{BiMonCumulants}
Let $(\A_{\ell}, \A_{r})$ be a pair of subalgebras in a non-commutative space $(\A, \varphi)$. The \textit{bi-monotonic cumulants} with respect to $(\A_\ell,\A_r,\varphi)$ is the family of functionals
\[\mathcal{K} = \left\{K_\chi:\A_{\chi(1)} \times \cdots \times \A_{\chi(n)} \to \bC\right\}_{n \geq 1, \chi: \{1, \dots, n\} \to \{\ell, r\}}\]
which satisfy
\begin{enumerate}[$\quad(1)$]
\item $K_\chi$ is multilinear,

\item there exists a polynomial $Q_\chi$ such that
\[\varphi(a_1\cdots a_n) = K_{\chi}(a_1, \dots, a_n) + Q_\chi\left(\{K_{\chi|_V}((a_1, \dots, a_n)|_V)\,|\,\emptyset \neq V \subsetneq \{1, \dots, n\}\}\right),\]

\item $K_{\chi}(N.a_1, \dots, N.a_n) = NK_{\chi}(a_1, \dots, a_n)$,
\end{enumerate}
for all $n \geq 1$ and $\chi: \{1, \dots, n\} \to \{\ell, r\}$.
\end{defn}

By the same arguments as in the proof of \cite{HS2011-2}*{Theorem 3.1}, it can be shown that bi-monotonic cumulants (if they exist) are unique. Indeed, for $n \geq 1$ and $\chi: \{1, \dots, n\} \to \{\ell, r\}$, Condition $(1)$ of Definition \ref{BiMonCumulants} implies that the polynomial $Q_\chi$ has no constant or linear terms. Moreover, Condition $(3)$ of Definition \ref{BiMonCumulants} implies that
\begin{align*}
\varphi(N.a_1\cdots N.a_n) &= K_\chi(N.a_1, \dots, N.a_n) +  Q_\chi\left(\{K_{\chi|_V}((N.a_1, \dots, N.a_n)|_V)\,|\,\emptyset \neq V \subsetneq \{1, \dots, n\}\}\right)\\
&= NK_{\chi}(a_1, \dots, a_n) + Q_\chi\left(\{NK_{\chi|_V}((a_1, \dots, a_n)|_V)\,|\,\emptyset \neq V \subsetneq \{1, \dots, n\}\}\right)\\
&= NK_{\chi}(a_1, \dots, a_n) + N^2\widetilde{Q}_\chi\left(\{N\} \cup \{K_{\chi|_V}((a_1, \dots, a_n)|_V)\,|\,\emptyset \neq V \subsetneq \{1, \dots, n\}\}\right),
\end{align*}
for some polynomial $\widetilde{Q}_\chi$. If $\K'$ is another family of functionals satisfying the same three conditions, then there is a polynomial $\widetilde{Q}'_\chi$ such that
\[\varphi(N.a_1\cdots N.a_n) = NK'_{\chi}(a_1, \dots, a_n) + N^2\widetilde{Q}'_\chi\left(\{N\} \cup \{K'_{\chi|_V}((a_1, \dots, a_n)|_V)\,|\,\emptyset \neq V \subsetneq \{1, \dots, n\}\}\right),\]
and hence $K_{\chi}(a_1, \dots, a_n) = K'_{\chi}(a_1, \dots, a_n)$.

Note also that by the recursive use of Condition $(2)$ of Definition \ref{BiMonCumulants}, there exists a polynomial $R_\chi$ such that
\[K_\chi(a_1, \dots, a_n) = \varphi(a_1\cdots a_n) + R_\chi\left(\left\{\varphi(a_V)\,|\,\emptyset \neq V \subsetneq \{1, \dots, n\}\right\}\right),\]
where $R_\chi$ also has no constant or linear terms by Condition $(1)$ of Definition \ref{BiMonCumulants}. This can be used as an equivalent condition. For the existence of bi-monotonic cumulants, we need the following analogue of \cite{HS2011-2}*{Proposition 3.2}.

\begin{prop}\label{DotMoment}
Let $(\A_{\ell}, \A_{r})$ be a pair of subalgebras in a non-commutative space $(\A, \varphi)$. For $n \geq 1$, $\chi: \{1, \dots, n\} \to \{\ell, r\}$ and $a_1, \dots, a_n \in \A$ with $a_j \in A_{\chi(j)}$, the moment $\varphi(N.a_1\cdots N.a_n)$ is a polynomial in
\[\{N\} \cup \left\{\varphi(a_V)\,|\,\emptyset \neq V \subsetneq \{1, \dots, n\}\right\}\]
without a constant term with respect to $N$.
\end{prop}

\begin{proof}
We proceed by induction on $n$ where the base case $n = 1$ is clear. For simplicity, we refer to $\left\{\varphi(a_V)\,|\,\emptyset \neq V \subsetneq \{1, \dots, n\}\right\}$ as the set of submoments. Note first that if $(\A'_\ell, \A'_r)$ and $(\A''_\ell,\A''_r)$ are bi-monotonically independent with respect to $\varphi$, then for $n \geq 1$, $\chi: \{1, \dots, n\} \to \{\ell, r\}$ and $a_1', \dots, a_n', a_1'', \dots, a_n'' \in \A$ with $a_j' \in \A'_{\chi(j)}$ and $a_j'' \in \A''_{\chi(j)}$, we have
\[\varphi((a'_1 + a''_1)\cdots(a'_n + a''_n)) = \varphi(a_1'\cdots a_n') + \varphi(a_1''\cdots a_n'') + S_\chi\left(\left\{\varphi(a'_V), \varphi(a''_V)\,|\,\emptyset \neq V \subsetneq \{1, \dots, n\}\right\}\right)\]
for some universal polynomial $S_\chi$ without a constant term (an explicit formula for $S_\chi$ is given in Lemma \ref{lem:FormulaTwoPairs}). Therefore, taking $a_i'= a_i^{(1)} + \cdots + a_i^{(N-1)}$ and $a_i''=a_i^{(N)}$, we have 
\begin{align*}
&\varphi((N + 1).a_1\cdots(N + 1).a_n) - \varphi(N.a_1\cdots N.a_n)\\
&= \varphi(a_1\cdots a_n) + S_\chi\left(\left\{\varphi(N.a_V), \varphi(a_V)\,|\,\emptyset \neq V \subsetneq \{1, \dots, n\}\right\}\right),
\end{align*}
where $\varphi(N.a_V) = \varphi(N.a_{v_1}\cdots N.a_{v_s})$ if $V = \{v_1 < \cdots < v_s\}$. Note that each monomial in the term $S_\chi(\bullet)$ above contains at least one factor $\varphi(N.a_V)$ and hence, by the induction hypothesis, is a polynomial in $N$ without a constant term and in the set of submoments. Hence
\begin{align*}
\varphi(N.a_1\cdots N.a_n) 
&=  \varphi(a_1\cdots a_n) + \sum_{L=1}^{N-1}\left[\varphi\left((L+1).a_1\cdots (L+1).a_n\right) - \varphi(L.a_1\cdots L.a_n)\right]  \\
&= \varphi(a_1\cdots a_n)+ \sum_{L = 1}^{N-1}\left[\varphi(a_1\cdots a_n) + S_\chi\left(\left\{\varphi(L.a_V), \varphi(a_V)\,|\,\emptyset \neq V \subsetneq \{1, \dots, n\}\right\}\right)\right]\\
&= N\varphi(a_1\cdots a_n) + \sum_{L = 1}^{N-1}S_\chi\left(\left\{\varphi(L.a_V), \varphi(a_V)\,|\,\emptyset \neq V \subsetneq \{1, \dots, n\}\right\}\right)
\end{align*}
is a polynomial in $N$ without a constant term and in the set of submoments since $1^p  + 2^p+\cdots (N-1)^p$ is a polynomial in $N$ without a constant term for each $p\geq1$. 
\end{proof}

\begin{prop}
Under the same assumptions and notation as Proposition \ref{DotMoment}, the coefficient of $N$ in $\varphi(N.a_1\cdots N.a_n)$ is the bi-monotonic cumulant $K_\chi(a_1, \dots, a_n)$ of $a_1, \dots, a_n$.
\end{prop}

\begin{proof}
Conditions (1) and (2) of Definition \ref{BiMonCumulants} follow from the proof of Proposition \ref{DotMoment} or from Theorem \ref{thm:moment-cumulant} below. For Condition $(3)$ of Definition \ref{BiMonCumulants}, we have
\[\varphi(M.(N.a_1)\cdots M.(N.a_n)) = \varphi((MN).a_1\cdots (MN).a_n)\]
by the associativity of bi-monotonic independence. Since $\varphi(M.(N.a_1)\cdots M.(N.a_n))$ can be written as
\[MK_\chi(N.a_1, \dots, N.a_n) + M^2Q'_\chi\left(\{M\} \cup \{\varphi(N.a_V)\,|\,\emptyset \neq V \subsetneq \{1, \dots, n\}\}\right)\]
for some polynomial $Q'_\chi$, and $\varphi((MN).a_1\cdots (MN).a_n)$ can be written as
\[MNK_\chi(a_1, \dots, a_n) + M^2N^2Q''_\chi\left(\{MN\} \cup \{\varphi(a_V)\,|\emptyset \neq V \subsetneq \{1, \dots, n\}\}\right)\]
for some polynomial $Q''_\chi$, we have $K_\chi(N.a_1, \dots, N.a_n) = NK_\chi(a_1, \dots, a_n)$.
\end{proof}

\begin{rem}
If we replace the associated independence of the dot operation by bi-free or bi-Boolean independence, then the bi-free and bi-Boolean cumulants can be defined by exactly the same procedure as above, which are unique and satisfy a stronger property (the vanishing characterization) than Condition $(3)$ of Definition \ref{BiMonCumulants}. For the explicit moment-cumulant formulae, see \cites{MN2015, CNS2015-1, GS2017}.
\end{rem}

\subsection{The bi-monotonic moment-cumulant formula}

In this subsection, a bi-monotonic moment-cumulant formula analogous to \cite{HS2011-2}*{Theorem 5.3} is described. To begin, the following definition is required.

\begin{defn}
Let $n \geq 1$ and let $\chi: \{1, \dots, n\} \to \{\ell, r\}$.
\begin{enumerate}[$\quad(1)$]
\item Denote by $\I(\chi)$ the set of subsets of $\{1, \dots, n\}$ which are $\chi$-intervals.

\item For $V = \{v_1 \prec_\chi \cdots \prec_\chi v_s\} \subset \{1, \dots, n\}$, denote by $V_\chi$ the set of $\chi$-intervals $\{V_1, \dots, V_{s + 1}\}$ where $V_1 = \{i \in \{1, \dots, n\}\,|\,i \prec_\chi v_1\}$, $V_{s + 1} = \{i \in \{1, \dots, n\}\,|\,v_s \prec_\chi i\}$, and $V_j = \{i \in \{1, \dots, n\}\,|\,v_{j - 1} \prec_\chi i \prec_\chi v_j\}$ for $2 \leq j \leq s$.
\end{enumerate}
\end{defn}

Building on Lemma \ref{lem:BiMonoFormula}, we have the following.

\begin{lem}\label{lem:FormulaTwoPairs}
Let $(\A_{\ell}', \A_{r}')$ and $(\A_{\ell}'', \A_{r}'')$ be pairs of algebras in a non-commutative space $(\A, \varphi)$ which are bi-monotonically independent with respect to $\varphi$. If $n \geq 1$, $\chi: \{1, \dots, n\} \to \{\ell, r\}$, and $a'_1, a''_1, \dots, a'_n, a''_n \in \A$ with $a'_j \in \A_{\chi(j)}'$ and $a''_j \in \A_{\chi(j)}''$, then
\[\varphi((a'_1 + a''_1)\cdots(a'_n + a''_n)) = \sum_{V \subset \{1, \dots, n\}}\varphi(a'_V)\prod_{W \in V_\chi}\varphi(a''_W).\]
\end{lem}

\begin{proof}
In the expansion of $\varphi((a'_1 + a''_1)\cdots(a'_n + a''_n))$, every term corresponds to a unique subset $V$ of $\{1, \dots, n\}$, where the elements of $V$ represent the positions of the elements from the pair $(\A_{\ell}', \A_{r}')$. The formula then follows from Lemma \ref{lem:BiMonoFormula}.
\end{proof}

As shown in Proposition \ref{DotMoment}, $\varphi(N.a_1\cdots N.a_n)$ is a polynomial in $N$ and the set of submoments, and thus we can replace $N$ by $t \in \bR$, denoted $\varphi_t(a_1, \dots, a_n)$, and obtain the following result.

\begin{cor}\label{MomentDiff}
Let $(\A_\ell,\A_r)$ be a pair of algebras in a non-commutative space $(\A, \varphi)$.  Then
\[\frac{d}{dt}\varphi_t(a_1, \dots, a_n) = \sum_{V \in \I(\chi)}\varphi_t\left((a_1, \dots, a_n)|_{V^\complement}\right)K_{\chi|_V}\left((a_1, \dots, a_n)|_V\right)\]
for all $n \geq 1$, $\chi: \{1, \dots, n\} \to \{\ell, r\}$, and $a_1, \dots, a_n \in \A$ with $a_j \in \A_{\chi(j)}$.
\end{cor}

\begin{proof}
Replacing $a_j'$ by $t.a_j$ and $a_j''$ by $s.a_j$ in Lemma \ref{lem:FormulaTwoPairs}, we obtain that
\[\varphi_{t + s}(a_1, \dots, a_n) = \sum_{V \subset \{1, \dots, n\}}\varphi_t((a_1, \dots, a_n)|_V)\prod_{W \in V_\chi}\varphi_s((a_1, \dots, a_n)|_W).\]
If we apply the derivation $\left.\frac{d}{ds}\right|_{s = 0}$ to the above equation, then each non-zero term on the right-hand side corresponds to a subset $V \subset \{1, \dots, n\}$ such that $|V_\chi| = 1$, i.e., $V_\chi$ consists of only one $\chi$-interval. It follows that
\[
\frac{d}{dt}\varphi_t(a_1, \dots, a_n) = \sum_{\substack{V \subset \{1, \dots, n\}\\V^\complement \in \I(\chi)}}\varphi_t((a_1, \dots, a_n)|_V)K_{\chi|_{V^\complement}}\left((a_1, \dots, a_n)|_{V^\complement}\right),
\]
from which the assertion follows by interchanging $V$ and $V^\complement$.
\end{proof}

To present the bi-monotonic moment-cumulant formula we require the following. 
\begin{defn}
Let $n \geq 1$ and let $\chi: \{1, \dots, n\} \to \{\ell, r\}$.
\begin{enumerate}[$\quad(1)$]
\item A \textit{bi-non-crossing partition with respect to $\chi$} is a partition $\pi$ on $\{1,\ldots, n\}$ such that $\pi$ is non-crossing with respect to the order $\prec_\chi$.  The set of all bi-non-crossing partitions is denoted by $\B\N\C(\chi)$.

\item A \textit{linearly ordered bi-non-crossing partition} (with respect to $\chi$) is a pair $(\pi, \lambda)$ where $\pi$ is a bi-non-crossing partition in $\B\N\C(\chi)$ and $\lambda$ is a linear ordering on the blocks of $\pi$. The set of all linearly ordered bi-non-crossing partitions is denoted by $\L\B\N\C(\chi)$.

\item If $\pi \in \B\N\C(\chi)$ and $V$ and $W$ are blocks of $\pi$, then $V$ is said to be \textit{interior} with respect to $W$ if there exist $w_1, w_2 \in W$ such that $w_1 \prec_\chi v \prec_\chi w_2$ for some (hence for all) $v \in V$.

\item A \textit{bi-monotonic partition} (with respect to $\chi$) is a linearly ordered bi-non-crossing partition $(\pi, \lambda) \in \L\B\N\C(\chi)$ with the following property: If $V$ and $W$ are blocks of $\pi$ such that $V$ is interior with respect to $W$, then $\lambda(W) < \lambda(V)$. The set of all bi-monotonic partitions is denoted by $\B\M(\chi)$.
\end{enumerate}
\end{defn}

The bi-free moment-cumulant formula (see \cites{MN2015, CNS2015-1}) is given by
\[\varphi(a_{1}\cdots a_{n}) = \sum_{\pi \in \B\N\C(\chi)}\left(\prod_{V \in \pi}\kappa_{\chi|_V}\left((a_1, \dots, a_n)|_V\right)\right)\]
for all $n \geq 1$, $\chi: \{1, \dots, n\} \to \{\ell, r\}$, $\omega: \{1, \dots, n\} \to K$, and $a_1, \dots, a_n \in \A$ with $a_j \in \A_{\omega(j), \chi(j)}$.  Alternatively, using the above definitions,  the bi-free moment-cumulant formula can be naturally written as
\[\varphi(a_{1}\cdots a_{n}) = \sum_{(\pi, \lambda) \in \L\B\N\C(\chi)}\frac{1}{|\pi|!}\left(\prod_{V \in \pi}\kappa_{\chi|_V}\left((a_1, \dots, a_n)|_V\right)\right),\]
where $|\pi|$ denotes the number of blocks of $\pi$.

For a bi-monotonic partition $(\pi, \lambda) \in \B\M(\chi)$, we observe that the largest (with respect to $\lambda$) block of $\pi$ must be a $\chi$-interval. If $V$ denotes this block, then $(\pi_0, \lambda_0) \in \B\M(\chi|_{V^\complement})$, where $\pi_0 = \pi \setminus \{V\}$, $\lambda_0$ denotes $\lambda$ restricted to $\pi_0$, and $\chi|_{V^\complement}$ denotes $\chi$ restricted to $\{1, \dots, n\} \setminus V$. Consequently, the sum $\sum_{(\pi, \lambda) \in \B\M(\chi)}$ can be written as $\sum_{V \in \I(\chi)}\sum_{(\pi_0, \lambda_0) \in \B\M(\chi|_{V^\complement})}$ by grouping together all bi-monotonic partitions with the same largest block.

\begin{thm}\label{thm:moment-cumulant}
Let $(\A_{\ell}, \A_{r})$ be a pair of algebras in a non-commutative space $(\A, \varphi)$. Then
\[\varphi(a_{1}\cdots a_{n}) = \sum_{(\pi, \lambda) \in \B\M(\chi)}\frac{1}{|\pi|!}\left(\prod_{V \in \pi}K_{\chi|_V}\left((a_1, \dots, a_n)|_V\right)\right)\]
for all $n \geq 1$, $\chi: \{1, \dots, n\} \to \{\ell, r\}$, and $a_1, \dots, a_n \in \A$ with $a_j \in \A_{\chi(j)}$.
\end{thm}

\begin{proof}
We proceed by induction on $n$ to show that
\[\varphi_t(a_1, \dots, a_n) = \sum_{(\pi, \lambda) \in \B\M(\chi)}\frac{t^{|\pi|}}{|\pi|!}\left(\prod_{V \in \pi}K_{\chi|_V}\left((a_1, \dots, a_n)|_V\right)\right)\]
for $t \in \bR$. The base case $n = 1$ is clear. For the inductive step, notice
\begin{align*}
&\sum_{(\pi, \lambda) \in \B\M(\chi)}\frac{t^{|\pi|}}{|\pi|!}\left(\prod_{V \in \pi}K_{\chi|_V}\left((a_1, \dots, a_n)|_V\right)\right)\\
&= \sum_{V \in \I(\chi)}\sum_{(\pi_0, \lambda_0) \in \B\M(\chi|_{V^\complement})}\frac{t^{|\pi_0| + 1}}{(|\pi_0| + 1)!}\left(\prod_{V_0 \in \pi_0}K_{\chi|_{V_0}}\left((a_1, \dots, a_n)|_{V_0}\right)\right)K_{\chi|_V}\left((a_1, \dots, a_n)|_V\right)\\
&= \sum_{V \in \I(\chi)}\int_0^t\sum_{(\pi_0, \lambda_0) \in \B\M(\chi|_{V^\complement})}\frac{s^{|\pi_0|}}{|\pi_0|!}\left(\prod_{V_0 \in \pi_0}K_{\chi|_{V_0}}\left((a_1, \dots, a_n)|_{V_0}\right)\right)K_{\chi|_V}\left((a_1, \dots, a_n)|_V\right)ds\\
&= \int_0^t\sum_{V \in \I(\chi)}\varphi_s\left((a_1, \dots, a_n)|_{V^\complement}\right)K_{\chi|_V}\left((a_1, \dots, a_n)|_V\right)\,ds\\
&= \int_0^t\frac{d}{ds}\varphi_s(a_1, \dots, a_n)\,ds\\
&= \varphi_t(a_1, \dots, a_n),
\end{align*}
where the third equality follows from the induction hypothesis and the fourth equality follows from Corollary \ref{MomentDiff}.
\end{proof}

Note that if $\chi: \{1, \dots, n\} \to \{\ell, r\}$ is constant, then $\B\M(\chi)$ reduces to the set $\M(n)$ of monotonic partitions of $\{1, \dots, n\}$ introduced in \cite{M2002} and used in \cite{HS2011-2}*{Section 5}, and Theorem \ref{thm:moment-cumulant} reduces to \cite{HS2011-2}*{Theorem 5.3}.

\begin{exam} We identify $\chi\colon\{1,\dots,n\}\to \{\ell,r\}$ with a sequence $(\chi(1),\dots,\chi(n))$. It is obvious that $K_{(\ell)}$ and $K_{(r)}$ are simply equal to $\varphi|_{\A_\ell}$ and $\varphi|_{\A_r}$, respectively. By the bi-monotonic moment-cumulant formula, we see that $K_{(\ell,\ell)},K_{(r,r)},K_{(\ell,r)}$ and $K_{(r,\ell)}$ are all the covariance. For the  cumulants of order three we can see a difference from the usual monotone cumulants. For $a_1,a_3 \in \A_\ell$ and $a_2\in\A_r$, Theorem \ref{thm:moment-cumulant} says that
\begin{align*}
\varphi(a_1a_2a_3) 
&= K_{(\ell,r,\ell)}(a_1,a_2,a_3) + \frac{1}{2}K_{(\ell,r)}(a_1,a_2) K_{(\ell)}(a_3) + K_{(\ell,\ell)}(a_1,a_3) K_{(\ell)}(a_2) \\
&\qquad+K_{(r,\ell)}(a_2,a_3) K_{(\ell)}(a_1) + K_{(\ell)}(a_1)K_{(r)}(a_2)K_{(\ell)}(a_3), 
\end{align*}
and hence 
\begin{align*}
K_{(\ell,r,\ell)}(a_1,a_2,a_3) 
&= \varphi(a_1a_2a_3) - \frac{1}{2}\Cov(a_1,a_2) \varphi(a_3) - \Cov(a_1,a_3) \varphi(a_2) \\
&\qquad- \Cov(a_2,a_3) \varphi(a_1) - \varphi(a_1)\varphi(a_2)\varphi(a_3) \\
&=  \varphi(a_1a_2a_3)- \frac{1}{2}\varphi(a_1a_2)\varphi(a_3) -\varphi(a_1a_3)\varphi(a_2) -\varphi(a_2a_3)\varphi(a_1) + \frac{3}{2} \varphi(a_1)\varphi(a_2)\varphi(a_3). 
\end{align*}
Similarly, for $a_1, a_2 \in\A_\ell$ and $a_3 \in \A_r$, 
\begin{align*}
K_{(\ell,\ell,r)}(a_1,a_2,a_3) 
&=  \varphi(a_1a_2a_3)- \varphi(a_1a_2)\varphi(a_3) - \frac{1}{2}\varphi(a_1a_3)\varphi(a_2) -\varphi(a_2a_3)\varphi(a_1) + \frac{3}{2} \varphi(a_1)\varphi(a_2)\varphi(a_3). 
\end{align*}

\end{exam}

\subsection{The bi-monotonic central, Poisson, and compound Poisson limit theorems}\label{subsec:LimitThms}

Using the bi-monotonic cumulants, we easily obtain the bi-monotonic version of central, Poisson, and compound Poisson limit theorems. For the monotonic central and Poisson limit theorems, see \cite{M2001}*{Sections 3 and 4} or \cite{HS2011-1}*{Section 5} for the cumulants approach.

Fix a non-commutative space $(\A, \varphi)$. The main focus here is the (joint) distribution of a pair $(a,b)$ of elements in a non-commutative space, which is the collection of moments $\{\varphi(c_1 c_2 \cdots c_n): n \in \bN, c_i \in\{a,b\}\}$. If $a$ and $b$ are commuting then the distribution reduces to $\{\varphi(a^m b^n): m,n \in \bN \cup \{0\}\}$. 
Consequently, if we define $\chi_{m, n}: \{1, \dots, m + n\} \to \{\ell, r\}$ by $\chi_{m, n}(k) = \ell$ if $k \leq m$ and $\chi_{m, n}(k) = r$ if $k > m$, and denote by $K_{m, n}(a, b)$ the bi-monotonic cumulant
\[K_{\chi_{m, n}}(\underbrace{a, \dots, a}_{m\,\mathrm{times}}, \underbrace{b, \dots, b}_{n\,\mathrm{times}}),\]
then it follows from Theorem \ref{thm:moment-cumulant} that the sequence $\{K_{m, n}(a, b)\}_{m, n \geq 0}$ also uniquely determines the joint distribution of $(a, b)$ with respect to $\varphi$. Limit theorems can now be stated as follows. Furthermore, recall that bi-monotonic independence implies commutativity (in distribution) of each left operator with each right operator from a different pair.
\begin{enumerate}[$\quad(1)$]
\item (The bi-monotonic central limit theorem) Let $\{(a_k, b_k)\}_{k = 1}^\infty$ be a sequence of identically distributed, commuting pairs in $(\A, \varphi)$ that are bi-monotonically independent with respect to $\varphi$ with $\varphi(a_1) = \varphi(b_1) = 0$, $\varphi(a_1^2) = \varphi(b_1^2) = 1$, and $\varphi(a_1b_1) := \gamma$. For $N \geq 1$, let
\[S_N^\ell = \frac{a_1 + \cdots + a_N}{\sqrt{N}} \qand S_N^r = \frac{b_1 + \cdots + b_N}{\sqrt{N}},\]
then $S_N^\ell$ commutes with $S_N^r$ for all $N \geq 1$. By the properties of bi-monotonic cumulants (see Definition \ref{BiMonCumulants}), we have $K_{1, 0}(S_N^\ell, S_N^r) = K_{0, 1}(S_N^\ell, S_N^r) = 0$, $K_{2, 0}(S_N^\ell, S_N^r) = K_{0, 2}(S_N^\ell, S_N^r) = 1$, $K_{1, 1}(S_N^\ell, S_N^r) = \varphi(a_1b_1) = \gamma$, and $K_{m, n}(S_N^\ell, S_N^r) = N^{-(m + n - 2)/2}K_{m, n}(a_1, b_1) \to 0$ as $N \to \infty$ for $m + n \geq 3$. Therefore, the sequence $\{(S_N^\ell, S_N^r)\}_{N = 1}^\infty$ converges in cumulants (hence in moments) as $N \to \infty$ to a commuting two-faced pair $(s_\ell, s_r)$ such that the only non-vanishing bi-monotonic cumulants are given by $K_{2, 0}(s_\ell, s_r) = K_{0, 2}(s_\ell, s_r) = 1$ and $K_{1, 1}(s_\ell, s_r) = \gamma$. We shall study this limiting object in greater detail using generating functions in the next section.

\item (The bi-monotonic Poisson limit theorem) Let $\lambda > 0$ and let $(0, 0) \neq (\alpha, \beta) \in \bR^2$. Suppose for every $N \geq 1$ that $\{(a_k^{(N)}, b_k^{(N)})\}_{k = 1}^N$ identically distributed, commuting pairs in $(\A, \varphi)$ that are bi-monotonically independent with respect to $\varphi$ such that
\begin{equation}\label{eqn:BMPoissonLimits}
\lim_{N \to \infty}N\varphi\left((a_1^{(N)})^m(b_1^{(N)})^n\right) = \lambda\alpha^m\beta^n
\end{equation}
for all $m + n \geq 1$. For $N \geq 1$, let
\[S_N^\ell = a_1^{(N)} + \cdots + a_N^{(N)} \qand S_N^r = b_1^{(N)} + \cdots + b_N^{(N)},\]
then $S_N^\ell$ commutes with $S_N^r$ for all $N \geq 1$. By the properties of bi-monotonic cumulants and the limits in \eqref{eqn:BMPoissonLimits}, we have
\[K_{m, n}(S_N^\ell, S_N^r) = NK_{m, n}(a_1^{(N)}, b_1^{(N)}) = N\left(\varphi\left((a_1^{(N)})^m(b_1^{(N)})^n\right) + O(1/N^2)\right) \to \lambda\alpha^m\beta^n\]
as $N \to \infty$ for all $m + n \geq 1$. Therefore, the sequence $\{(S_N^\ell, S_N^r)\}_{N = 1}^\infty$ converges in distribution as $N \to \infty$ to a commuting two-faced pair $(s_\ell, s_r)$ with bi-monotonic cumulants given by $K_{m, n}(s_\ell, s_r) = \lambda\alpha^m\beta^n$ for $m + n \geq 1$.

\item (The compound bi-monotonic Poisson limit theorem) Let $\lambda >0$ and let $\nu$ be a probability measure on $\bR^2$ with compact support such that $\nu(\{(0,0)\}) = 0$. In the setting of the previous example, suppose more generally that the distribution of $(a_1^{(N)}, b_1^{(N)})$ is such that 
\[ \lim_{N\to\infty} N\varphi\left((a_1^{(N)})^m(b_1^{(N)})^n\right)=\lambda M_{m, n}(\nu) := \lambda\int_{\bR} s^m t^n d\nu(s,t), \quad m, n \geq 0.\]
For example we may take the moments of $(a_1^{(N)}, b_1^{(N)})$ to be those of 
\[
\left(1-\frac{\lambda}{N}\right)\delta_{(0,0)} + \frac{\lambda}{N}\nu.
\]
The previous technique shows the convergence 
\[K_{m, n}(S_N^\ell, S_N^r) = NK_{m, n}(a_1^{(N)}, b_1^{(N)}) \to \lambda M_{m, n}(\nu)\]
as $N \to \infty$. With the moment-cumulant formula, this implies that the moments $\varphi\left((S_N^\ell)^m(S_N^r)^n\right)$ converge to the numbers 
\[
M_{m,n} = \sum_{\pi \in \B\M(\chi_{m,n})} \frac{\lambda^{|\pi|}}{|\pi|!}  \prod_{V\in\pi} M_{\ell(V), r(V)}(\nu), 
\]
where $\ell(V) = |V \cap\{1,\dots,m\}|$ and $r(V)=|V \cap \{m+1,\dots, m+n\}|$. We will show in Section \ref{sec:analytic} that $M_{m,n}$ are not moments of a probability measure in general. 
\end{enumerate}

\subsection{Generating function of bi-monotonic cumulants of single variable} \label{subsec:GF-cumulants} In monotonic probability, a differential equation relates a generating function of monotonic cumulants to a moment generating function \cite{HS2011-2}. We give more general differential equations in the bi-monotonic setting. 
 Let $(a, b)$ be a commuting pair in a non-commutative space $(\A, \varphi)$. For $N \in \bN$, define
\[
G_{1, N}(z) = \sum_{m \geq 0}\frac{\varphi((N.a)^m)}{z^{m + 1}}, \quad G_{2, N}(w) = \sum_{n \geq 0}\frac{\varphi((N.b)^n)}{w^{n + 1}}, \qand G_N(z, w) = \sum_{m, n \geq 0}\frac{\varphi((N.a)^m(N.b)^n)}{z^{m+1}w^{n+1}},
\]
where the dot operation is associated with the bi-monotonic independence. As seen above, $\varphi((N.a)^m(N.b)^n)$ is a polynomial in $N$, thus we may replace $N$ by $t \in \bR$ and obtain a formal power series $G_t(z,w)$. Notice we do not know whether $G_t(z,w)$ is convergent for large $|z|$ and $|w|$, and hence it is defined only as a formal power series. Similarly, define $G_{1, t}$ and $G_{2, t}$ by replacing $N$ by $t$ and denote their reciprocals by $F_{1, t}$ and $F_{2, t}$, respectively. Since the coefficient of $N$ in $\varphi((N.a)^m(N.b)^n)$ is the bi-monotonic cumulant $K_{m, n}(a, b)$, the cumulant generating series of $(a, b)$ is given by
\[
\widetilde{A}(z,w) = (z w)\frac{d}{dt}\left[G_t(z, w)\right]\bigg|_{t=0} = \sum_{\substack{m, n \geq 0\\m + n \geq 1}}\frac{K_{m, n}(a, b)}{z^mw^n}.
\]
Moreover, the convolution formula \eqref{eqn:biMonCauchy} implies
$
G_{M+N}(z,w) = G_M(F_{1,N}(z), F_{2,N}(w)) G_ N(z,w) F_{1,N}(z) F_{2,N}(w), 
$
and replacing $(M,N)$ by $(s,t)$ produces
\begin{equation}\label{eq conv-semigroup}
G_{s+t}(z,w) = G_s(F_{1,t}(z), F_{2,t}(w))G_ t(z,w)F_{1,t}(z)F_{2,t}(w).
\end{equation}

It is convenient to separate the marginal parts and correlation part of $\widetilde{A}$ by letting
\[
H_t(z,w) = G_{t}(z,w)F_{1,t}(z)F_{2,t}(w).
\]
Then Equation \eqref{eq conv-semigroup} reads 
\begin{equation}\label{eq conv-semigroup2}
H_{s+t}(z,w) = H_s(F_{1,t}(z), F_{2,t}(w)) H_t(z,w), \quad H_0(z,w) = 1,
\end{equation}
and we know from the single-variable case that
\begin{equation}\label{diffeq Cauchy}
G_{j, s+t}(z) = G_{j,s}(F_{j,t}(z)), \quad j = 1, 2.
\end{equation}
Define the marginal parts of $\widetilde{A}$ by
\begin{align*}
A_1(z) &= z^2\frac{d}{dt}\left[G_{1,t}(z)\right]\bigg|_{t=0} = \sum_{m \geq 1}\frac{K_{m, 0}(a,b)}{z^{m-1}},\\
A_2(w) &= w^2\frac{d}{dt}\left[G_{2,t}(w)\right]\bigg|_{t=0} = \sum_{n \geq 1}\frac{K_{0, n}(a,b)}{w^{n-1}},
\end{align*}
and the correlation part of $\widetilde{A}$ by
\[
A(z,w) = \frac{d}{dt}\left[H_t(z,w)\right]\bigg|_{t=0} = \sum_{m, n \geq 1}\frac{K_{m,n}(a,b)}{z^mw^n}.
\]
Then 
\begin{equation}\label{GF-cumulants}
\widetilde{A}(z,w) = \frac{1}{z}A_1(z) + \frac{1}{w}A_2(w) + A(z,w). 
\end{equation}

Finally, we notice that taking the derivatives of Equations \eqref{eq conv-semigroup}, \eqref{eq conv-semigroup2}, and \eqref{diffeq Cauchy} with respect to $s$ at $0$ yield the following differential equations:
\begin{equation}\label{eq conv-semigroup3}
\frac{\partial}{\partial t}G_t(z,w) = G_t(z,w)\widetilde{A}(F_{1,t}(z), F_{2,t}(w)), \quad G_0(z,w) = \frac{1}{z w},
\end{equation}
\[
\frac{\partial}{\partial t}H_t(z,w) = H_t(z,w)A(F_{1,t}(z), F_{2,t}(w)), \quad H_0(z,w) = 1,\]
\begin{equation}\label{eq conv-semigroup4}
\frac{\partial}{\partial t}F_{j,t}(z) = -A_j(F_{j,t}(z)), \quad F_{j, 0}(z) = z, \quad j = 1, 2.
\end{equation}
The last equation \eqref{eq conv-semigroup4} was previously obtained in \cite{HS2011-2}*{Equation 6.2}.

\begin{exam}\label{exam:CLT}
We revisit the bi-monotonic central limit distribution from Subsection \ref{subsec:LimitThms}. Suppose $(a, b)$ has bi-monotonic cumulants $K_{2, 0}(a, b) = \alpha$, $K_{0, 2}(a, b) = \beta$,  $K_{1, 1}(a, b) = \gamma$, and other cumulants being zero, where $\alpha, \beta > 0$ and $|\gamma| \leq \sqrt{\alpha\beta}$. In this case,
\[A_1(z) = \frac{\alpha}{z}, \quad A_2(w) = \frac{\beta}{w}, \qand A(z, w) = \frac{\gamma}{zw},\]
thus Equation \eqref{eq conv-semigroup4} yields
\[F_{1, t}(z) = \sqrt{z^2 - 2\alpha t} \qand F_{2, t}(w) = \sqrt{w^2 - 2\beta t}, \]
where $F_{1, t}$ is defined so that it is analytic in $\bC \setminus [-\sqrt{2\alpha t}, \sqrt{2\alpha t}]$ and $F_{1, t}(z) =z(1+o(1))$ as $z\to \infty$, and similarly for $F_{2, t}.$ 
Solving the differential equation \eqref{eq conv-semigroup3} now produces the solution
\begin{align*}
G_t(z,w) &= \frac{1}{zw}\exp\left[\int_0^t\widetilde{A}(F_{1,s}(z), F_{2,s}(w))\,ds\right]\\
&= \frac{1}{zw}\exp\left[\int_0^t\left(\frac{\alpha}{z^2-2\alpha s} + \frac{\gamma}{\sqrt{z^2-2\alpha s}\sqrt{w^2-2\beta s}} + \frac{\beta}{w^2-2\beta s}\right)\,ds\right]\\
&= \frac{1}{\sqrt{z^2 - 2\alpha t}\sqrt{w^2 - 2\beta t}}\left(\frac{\sqrt{\beta z^2 - 2\alpha\beta t} - \sqrt{\alpha w^2 - 2\alpha\beta t}}{\sqrt{\beta} z - \sqrt{\alpha}w} \right)^{\frac{\gamma}{\sqrt{\alpha\beta}}},
\end{align*}
where the last power function is defined analytically in $(\bC\setminus\bR)^2$ so that it converges to $1$ as $z \to \infty$ or $w\to\infty$. The Cauchy transform of $(a, b)$ is then given by $G_1(z, w)$. 
\end{exam}

\begin{exam}
Suppose $(a, b)$ has a bi-monotonic Poisson distribution with $K_{m, n}(a, b) = \lambda \alpha^m\beta^n$, where $\lambda > 0$ and $(0,0) \neq (\alpha, \beta) \in \bR^2$. In this case,
\[A_1(z) = \frac{\lambda\alpha z}{z-\alpha}, \quad A_2(w) = \frac{\lambda\beta w}{w-\beta}, \quad A(z, w) = \frac{\lambda \alpha \beta}{(z-\alpha)(w-\beta)},\]
and 
\[
\widetilde A(z,w) = \lambda \left(\frac{ z w}{(z-\alpha)(w-\beta)}-1\right). 
\]
Muraki showed that the functions $G_{1,t}$ and $G_{2,t}$ can be expressed by using the Lambert $W$ function \cite{M2001}. We do not know how the Cauchy transform $G_t$ can be expressed.
\end{exam}

\begin{exam}\label{exam:compound}
Suppose $(a, b)$ has a compound bi-monotonic Poisson distribution with $K_{m, n}(a, b) = \lambda M_{m,n}(\nu)$, where $\lambda >0$ and $\nu$ is a probability measure on $\bR^2$ with compact support such that $\nu(\{(0,0)\})=0$. In this case, we can obtain 
\[\widetilde A(z,w) =\lambda \int_{\bR^2}\left( \frac{z w}{(z-s)(w-t)}-1 \right)d\nu(s,t).\]

\end{exam}

\section{Non-existence of bi-monotonic convolution of probability measures}\label{sec:analytic}

\subsection{Bi-monotonic convolution}

As Example \ref{not-state} shows, the bi-monotonic product of states is not a state in general.  However this does not quickly imply that a bi-monotonic convolution does not exist for probability measures on $\mathbb R^2$. We demonstrate that such a convolution indeed does not exist.

Let $(a_1, b_1)$ and $(a_2, b_2)$ be two pairs of commuting elements that are bi-monotonically independent with respect to $\varphi$ and whose distributions with respect to $\varphi$ are probability measures. 
For the distribution of $(a_1+a_2,b_1+b_2)$ to be a probability measure on $\mathbb R^2$, it is necessary that  for every $n \in \bN$ and every polynomial $p(x, y) = \sum^n_{k,m = 0} c_{k,m} x^k y^m$ where $c_{k,m} \in \bC$ we have 
\[
0 \leq \varphi(p(a_1+a_2, b_1+b_2)^*p(a_1+a_2, b_1+b_2)) = \sum^n_{i_1, i_2, j_1, j_2 = 0} c_{i_1, i_2} \overline{c_{j_1, j_2}} \varphi((a_1+a_2)^{i_1+j_1} (b_1+b_2)^{i_2+j_2}).
\]
Therefore, for a fixed $n$ we consider the $(n+1)^2 \times (n+1)^2$ matrix 
\[
X_n = [\varphi((a_1+a_2)^{i_1+j_1} (b_1+b_2)^{i_2+j_2})]
\]
where the rows are indexed (starting at 0 up to $n$) by the pairs $(i_1, i_2)$ and the columns are indexed by the pairs $(j_1, j_2)$.  Then
\[
\varphi(p(a_1+a_2, b_1+b_2)^*p(a_1+a_2, b_1+b_2)) = \langle X_n\vec{v}, \vec{v}\rangle, 
\]
where $\vec{v}$ is a vector of the $c_{k,m}$'s. For example the matrix $X_1$ is given by 
\[
X_1 = \begin{bmatrix}
1 & \varphi(a_1+a_2) & \varphi(b_1 + b_2) & \varphi((a_1+a_2)(b_1+b_2)) \\
\varphi(a_1 + a_2) & \varphi((a_1+a_2)^2) & \varphi((a_1+a_2)(b_1+b_2)) & \varphi((a_1+a_2)^2(b_1+b_2)) \\
\varphi(b_1 + b_2) & \varphi((a_1+a_2)(b_1+b_2)) & \varphi((b_1 + b_2)^2)  & \varphi((a_1+a_2)(b_1+b_2)^2) \\
 \varphi((a_1+a_2)(b_1+b_2)) &  \varphi((a_1+a_2)^2(b_1+b_2)) &  \varphi((a_1+a_2)(b_1+b_2)^2) &  \varphi((a_1+a_2)^2(b_1+b_2)^2)
\end{bmatrix}. 
\]

Assume that $(a_1,b_1)$ and $(a_2,b_2)$ have the same probability distribution $\mu = \frac{1}{2}(\delta_{(0,1)}+\delta_{(1,0)})$. Notice that 
\[
\int_{\bR^2} x^m y^n \, d\mu(x,y) = \begin{cases}
1 & \text{if }m= n=0, \\
\frac{1}{2} & \text{if }(m,n) \in \{(k, 0), (0, k) \, \mid \, k \in \bN\}, \\
0 & \text{otherwise}. 
\end{cases}
\]
We can use either the formula \eqref{eqn:biMonCauchy} or the direct definition of bi-monotonic independence, to obtain 
\[
X_1 = \begin{bmatrix}
1 & 1 & 1 & \frac{1}{2} \\
1 & \frac{3}{2} & \frac{1}{2} & \frac{5}{8} \\
1 & \frac{1}{2} & \frac{3}{2} & \frac{5}{8} \\
\frac{1}{2} & \frac{5}{8} & \frac{5}{8} & \frac{3}{4}
\end{bmatrix}.
\]
Hence $X_1$ is clearly a self-adjoint matrix.  However, one can check that $\det(X_1) = - \frac{1}{32}$.  This implies $X_1$ is not positive semidefinite and thus the distribution of $(a_1+a_2, b_1+ b_2)$ is not a probability measure.

\subsection{Compound bi-monotonic Poisson distributions}
To further illustrate the lack of a bi-monotonic convolution theory for probability measures, we show that the compound bi-monotonic Poisson distribution is not a probability measure in general. 
Consider a pair of elements $(a,b)$ in a non-commutative space that has a compound bi-monotonic Poisson distribution characterized by 
$$
K_{m,n}\equiv K_{m,n}(a,b)= \int_{\mathbb R^2}s^m t^n\, d\tau(s,t), 
$$ 
where $\tau = 15\delta_{(1,1)}+ 15\delta_{(-1,1)}+15\delta_{(1,-1)}$. Then $K_{m,n}=15(-1)^m + 15(-1)^n +15$ for $m,n \geq0, (m,n) \neq (0,0)$. 
We compute the moments of $(a,b)$ using the differential equations derived in Section \ref{subsec:GF-cumulants}. Solving the differential equation $\partial_t F_{j,t}(z) = -A_j(F_{j,t}(z)), F_{j,0}(z)=z$ (or using the monotonic moment-cumulant formula), we obtain 
\begin{align*}
G_{1,t}(1/z) = G_{2,t}(1/z) 
&= z + K_{1,0} t z^2 +  ( K_{2,0} t +K_{1,0}^2 t^2)z^3+ \cdots\\  
&= z + 15 t z^2 +  (45t+225t^2)z^3+ \cdots.  
\end{align*}
Hence 
\begin{align*}
&\int_0^1\tilde{A}(F_{1,s}(1/z), F_{2,s}(1/w))\, ds  \\
&= \int_0^1 \left( K_{1,0} G_{1,s}(1/z) + K_{0,1} G_{2,s}(1/w) + K_{2,0} G_{1,s}(1/z)^2 + K_{1,1}G_{1,s}(1/z) G_{2,s}(1/w) + K_{0,2} G_{2,s}(1/w)^2 \right.  \\
&\quad\left. + K_{2,1} G_{1,s}(1/z)^2 G_{2,s}(1/w) +K_{1,2}G_{1,s}(1/z) G_{2,s}(1/w)^2 + K_{2,2}G_{1,s}(1/z)^2 G_{2,s}(1/w)^2 \right)ds \\
&=15z + 15w +\frac{315 z^2}{2} - 15z w +\frac{315 w^2}{2} -\frac{195z^2 w}{2} -\frac{ 195z w^2}{2}-855 w^2 z^2 \\
& \quad+\text{~the sum of terms $z^m w^n$ with $m >2$ or $n>2$}.
\end{align*}
Then
\begin{align*}
G_{(a,b)}(1/z,1/w) &= z w \exp  \int_0^1\tilde{A}(F_{1,s}(1/z), F_{2,s}(1/w))\, ds \\
&=z w + 15z^2 w + 15z w^2 + 270 z^3 w + 210 z^2w^2+270z w^3+\frac{7455 z^3 w^2}{2} + \frac{7455 z^2w^3}{2} \\ &\qquad+\frac{131715 z^3 w^3}{2} +\text{~the sum of terms $z^m w^n$ with $m >3$ or $n>3$}. 
\end{align*}
Denote by $M_{m,n}$ the coefficient of $z^{m+1} w^{n+1}$ in the above expression of $G(1/z,1/w)$, which are the moments of our compound bi-monotonic Poisson distribution. The determinant of $4 \times 4$ matrix $(M_{m_1+ m_2,n_1+n_2})$, with the 4 pairs $\{(m_1, n_1): m_1, n_1 =0,1\}$ as the index of column and 4 pairs $\{(m_2,n_2): m_2,n_2=0,1\}$ as the index of row, is 
$$
\left|\begin{matrix} 
1 & 15 & 15 & 210 \\
15 & 270 & 210 & 7455/2 \\ 
15 & 210 & 270 & 7455/2 \\
210 & 7455/2 & 7455/2 & 131715/2
\end{matrix}\right|
=-857250 <0. 
$$
Hence $M_{m,n}$ are not moments of a probability measure on $\mathbb R^2$.

\section*{Acknowledgements} 
T.H.\ is supported by JSPS Grant-in-Aid for Young Scientists (B) 15K17549 and (A) 17H04823. P.S.\ is supported by NSERC (Canada) grant RGPIN-2017-05711.  The authors are grateful to Malte Gerhold for pointing out an error of the previous manuscript.

\end{document}